%% file: main.tex
\documentclass{amsart}

\input{preamble}

\begin{document}
\title{Average Distortion of Commensurators of Hyperbolic Groups}

\author{Nir Lazarovich}

\address{Department of Mathematics
	Technion
	Haifa 32000
	Israel}
 \email{lazarovich@technion.ac.il}
 
\author{Suraj Krishna M S}

\address{Department of Mathematics, Ashoka University
Rai Sonepat Haryana, 131029, India}

\email{suraj.meda@ashoka.edu.in}

\author{Mahan Mj}

\address{School of Mathematics, Tata Institute of Fundamental Research, Mumbai-40005, India}
\email{mahan@math.tifr.res.in}
\email{mahan.mj@gmail.com}
\urladdr{http://www.math.tifr.res.in/~mahan}
\date{\today}


\thanks{NL was partially supported by the Israeli Science Foundation (grant no. 1576/23). 
SKMS was supported by the Israeli Science Foundation (grant no. 1226/19) at the Technion, an annual research grant of Ashoka University, and by an ANRF-MATRICS grant (ANRF/ARGM/2025/000741/MTR).
MM was supported by the Department of Atomic Energy,
Government of India, under project no.12-R\& D-TFR-5.01-0500 as also by an endowment of the
Infosys Foundation.  }

\begin{abstract}
We prove that commensurators of a geometrically rigid residually finite hyperbolic group have bounded average distortion.
\end{abstract}

\maketitle



\section{Introduction}

Let $G$ be a one-ended torsion-free hyperbolic group. 
Paulin \cite{paulin1991outer}, building on the Rips machine (cf. Bestvina--Feighn \cite{bestvina1995stable}), showed that  $|\Out(G)|<\infty$ if and only if $G$ does not split over a virtually cyclic group.
 A  problem motivating this paper is to try to generalize this result to commensurators of $G$, i.e. to isomorphisms between finite index subgroups of $G$.

A one-ended hyperbolic group for which $\Out(G)$ is finite is called  \emph{rigid}. 
Bowditch \cite{bowditchjsj} provides a geometric characterization of rigidity: a one-ended hyperbolic group is rigid if and only if either (1) $G$ surjects to a triangle group with finite kernel or (2) the Gromov boundary $\partial G$ is connected and does not contain a cut pair of points.
If $G$ is as in (1), then it is virtually a surface group, and hence virtually non-rigid. 
We therefore focus on case (2) which we call \emph{geometrically rigid}. 
We remark that if $G$ is torsion-free, the two notions are the same.

Trying to generalize Paulin's result to commensurators,  we run into two problems:
First, one cannot define an ``outer commensurator group'', since the inner automorphisms do not form a normal subgroup in $\Comm(G)$. 
Second, since $|\Out(G)| < \infty$ whenever $G$ is geometrically rigid, one can
still consider the index $[\Comm(G):G]$
for such groups. However, 
there are (geometrically) rigid groups for which $[\Comm(G):G] = \infty$.
For example, arithmetic uniform lattices in $\SO(n,1)$ for $n\ge 3$ are such.
This is a more substantial problem.

Let us elaborate more on this last example:
Let $M$ be a hyperbolic $n$-manifold ($n\ge 3$) and let $G$ be its fundamental group. 
Let $\phi:H_1\to H_2$ be an isomorphism between finite index subgroups of $G$. 
Then, by Mostow Rigidity \cite{mostow1968quasi} $\phi$ is induced by an isometry $\overline \Phi:M_1\to M_2$ between the finite covers $M_1,M_2$ of $M$ corresponding to $H_1,H_2$.
Alternatively, the lift $\tild \Phi:\bbH^n \to \bbH^n$ of $\overline \Phi$ to the universal cover $\bbH^n$ of $M$ is a $\phi$-equivariant isometry of $\bbH^n$.

Since in this paper we work with countable/discrete groups, we would like to phrase the existence of $\tild \Phi$ more intrinsically in terms of $G$. 
This can be done by approximating $\tild \Phi : \bbH^n \to \bbH^n$ by a $\phi$-equivariant map $\Phi: G\to G$ which is a $(C,C)$-quasi-isometry (with respect to the word metric on $G$).
Importantly, the constant $C$ here is \textbf{independent} of the map $\phi$ (but might depend on  $G$ and its word metric).
We conjecture that this is indeed the case for all geometrically rigid hyperbolic groups:

\begin{conjecture}
If $G$ is a geometrically rigid hyperbolic group, then there exist a constant $C=C(G)$ such that every abstract commensurator $\phi:H_1\to H_2$ admits a $\phi$-equivariant map $\Phi:G\to G$ which is $(C,C)$-quasi-isometry.
\end{conjecture}

This conjecture can be seen as an interpretation of 0.3.C in Gromov's \emph{Hyperbolic Groups} \cite{gromov1987hyperbolic}, and has interesting implications.
In this paper we show that a variant of this conjecture holds -- namely, the ``average distortion'' of commensurators is uniformly bounded.

\subsubsection*{Average Distortion of commensurators}
Let $G$ be a hyperbolic group.
Let $X$ be a graph on which $G$ acts freely and cocompactly, e.g., its Cayley graph with respect to some finite generating set $S$. 
Let $\phi:H\to G$ be an injective map where $H\le G$ has finite index. Consider a $\phi$-equivariant (cellular) map $\Phi:X\to X$ sending vertices to vertices and edges to edge-paths.
We denote by $\ell(\Phi(e))$ the length of the edge-path $\Phi(e)$.
Assume $e_1,\dots,e_k$ are representatives of the $H$-orbits of edges in $X$ (note that $k = k_0\cdot [G:H]$ where $k_0$ is the number of edge $G$-orbit representatives).
Define the \emph{total distortion} of $\Phi$ by $$\TD(\Phi) = \sum_{i=1}^k \ell(\Phi(e_i))$$ and the \emph{average distortion} of $\Phi$ by $$ \AD(\Phi) = \tfrac 1{k}\TD(\Phi)  \propto \tfrac 1{[G:H]}\TD(\Phi).$$
Denote $\AD(\phi) = \min_\Phi \AD(\Phi)$ where the minimum is taken over all $\phi$-equivariant cellular maps $\Phi:X\to X$.

\begin{restatable}{theoremA}{maintheorem}\label{main theorem}
    Let $G$ be a geometrically rigid, residually finite, hyperbolic group. 
    Let $X$ be a graph with a free and cocompact $G$-action.
    There exists a finite index subgroup $\dot G\le G$ and a constant $C\ge 0$ such that for every finite index subgroup $H\le \dot G$ and every
    injective homomorphism $\phi:H\to G$, we have 
    
    \begin{equation}\label{the right inequality}\AD(\phi)\le C.
    \end{equation}
    If, moreover, the image of $\phi$ has finite index in $G$ then 
    \begin{equation}\label{both inequalities for AD}
    \frac1C\le \AD(\phi)\le C.
    \end{equation}
    \end{restatable}

    We note that the existence of a constant $C$ such that \eqref{both inequalities for AD} holds is equivalent to the existence of a constant $C'$ such that 
    \begin{equation}\label{main theorem inequality}
    \frac{1}{C'}\cdot [G:H]\le \TD(\Phi)\le {C'}\cdot [G:H].
    \end{equation}
    for the map $\Phi$ that minimizes $\TD(\Phi)$ among all $\phi$-equivariant maps $\Phi:X\to X$.
    We also note that while the constant $C$ depends on $X$, the existence of such a constant does not (see \Cref{all graphs are equivalent}).

\subsubsection*{Commensurators of negatively curved manifolds}

\begin{corollary}\label{main theorem manifolds}
Let $M$ be a closed manifold of dimension $\ge 3$ of negative sectional curvature with residually finite fundamental group. Then, there exists $C\ge 1$ and a cover $\dot M$ such that any homotopy equivalence $F:M_1\to M_2$ between finite covers of $\dot M$ is homotopic to a differentiable map $f:M_1\to M_2$ such that
\begin{equation}\label{manifold inequality}
    \tfrac 1 {\Vol(M_1)}\int_{T^1M_1}\|D_x f(v)\|\; \bfd \vol(x,v) \le C
\end{equation}
\end{corollary}

Conjecturally, there exists $f$ that satisfies $\tfrac 1C\le  \|D_xf(v)\|\le C$ for all $(x,v)\in T^1M_1$. 
This holds if $M$ has constant negative sectional curvature, since by Mostow Rigidity then $f$ can be chosen to be an isometry, i.e. $\|D_xf(v)\|=1$.

\subsubsection*{Outer automorphisms}

The main theorem gives immediately an alternative proof of Paulin's Theorem (under the additional assumption of residual finiteness) that does not use the Rips machine.

\begin{corollary}[Paulin's Theorem for residually finite]\label{paulin thm}
If $G$ is a  residually finite geometrically rigid hyperbolic group then $\Out(G)$ is finite.
\end{corollary}

\subsubsection*{Average distortion of translation length}
For $g\in G$ let $$\ell(g) = \lim_{n\to \infty} \tfrac1n \|g^n\|$$ be its stable translation length. We would like to define the translation length of ``$\phi(g)$'' for some isomorphism $\phi:H\to H'$ between finite index subgroups. However, $\phi$ might not be defined on $g$.
To remedy this, we need to lift $g$ to $H$. 

\begin{definition}
    The \emph{lift of $g$ to $H$} is the power $g^m$ for the minimal $m\in \bbN$ such that $g^m\in H$. The exponent $m$ is called the \emph{lifting index of $g$ to $H$}.
\end{definition}

Given $\phi:H\to H'$ we define $$\ell(\phi,g) = \tfrac 1m\ell (\phi(g^m))$$
Unlike the stable translation length, $\ell(\phi,g)$ is not invariant under conjugation by elements of $G$, but it is invariant under conjugation in $H$. 
To fix this, let $\{t_i\}_{1\le i\le[G:H]}$ be a set of representatives of $G/H$ and define  $$\AL(\phi,g) = \frac{1}{[G:H]}\sum_{i=1}^{[G:H]}\ell(\phi,t_i\ii gt_i)$$

\begin{corollary}\label{cor: AL}
    Let $G$ be a geometrically rigid, residually finite hyperbolic group. 
    Then there exists $C$ such that for every finite index subgroup $H\le G$, injective isomorphism $\phi:H\to G$ and element $g\in G$ we have $\AL(\phi,g)\le C\cdot \ell (g)$.
\end{corollary}

We note that both $\ell(\phi,g)$ and $\AL(\phi,g)$ are invariant under the following equivalence relation: Let $\phi:H\to G$ and $\psi:K\to G$ be injective maps from finite index subgroups of $G$. We write $\phi\sim \psi$ if there exists some finite index subgroup $L$ such that $\phi|_L$ and $\psi|_L$ are defined and equal.
We recall that the \emph{abstract commensurator} $\Comm(G)$ of $G$ is the group of all equivalence classes of isomorphisms between finite index subgroup of $G$, under composition.

\subsubsection*{Relative commensurators of hyperbolic subgroups of hyperbolic groups}

For a subgroup $G\le G_1$, the \emph{(relative) commensurator of $G$ in $G_1$} is the subgroup $$\Comm_{G_1}(G)=\{\gamma\in G_1\;|\;\gamma G \gamma\ii \cap G\text{ has finite index in }G\text{ and }\gamma G\gamma\ii\}$$

We prove the following variant of the main result of \cite{lazarovich2023commensurated} (under the additional assumption of residual finiteness) 
\begin{corollary}[see {\cite[Theorem A]{lazarovich2023commensurated}}]\label{cor commensurated hyperbolic subgroup} Let $G\le G_1$ be hyperbolic groups, and assume that $G$ is residually finite. If $[\Comm_{G_1}(G):G]=\infty$, then $G$ is virtually a free product of surface and free groups.
\end{corollary}

\subsubsection*{Outline of the paper and of the proof of \cref{main theorem}}

Let $G$ be a residually finite, geometrically rigid hyperbolic group, let $\phi:H\to H'$ be an isomorphism between finite index subgroups and let $\Phi:X\to X$ be a $\phi$-equivariant cellular map on the Cayley graph $X$ of $G$. 

In \Cref{sec: singular patterns}, we will use the globally stable cylinders of Petyt-Spriano-Zalloum \cite{petyt2025stable} and ideas of Delzant \cite{delzant1995image} to construct a singular pattern $\calF$ on the Rips complex $Y$ of $X$ (with respect to some constant $D$), and a singular pattern $\overline\calF$ on its quotient $\overline Y :=Y/H$.
This pattern is simply a graph that is immersed in $\overline Y$. 
Each connected component $\overline\lambda \subseteq \overline\calF$ can be thought of as a compactly supported 1-cochain in $Y$.
We show the following inequalities
\[
    \TD(\Phi)\preceq\sum_{\lambda \subseteq \overline\calF}|\lambda| \quad \text {and} \quad \sum_{\lambda \subseteq \overline\calF} |\bfd \lambda| \preceq [G:H].
\]
where the sums run over the connected components $\lambda$ of $\overline\calF$,  and $|\lambda|$ and $|\bfd\lambda|$ denote the size of the support of $\lambda$ and its coboundary respectively.

Thus, to prove the right inequality of \eqref{main theorem inequality}, it suffices to show $|\lambda|\preceq |\bfd\lambda|$.
If $\lambda$ is minimal, i.e. it has the minimal size among all 1-cochains with the same coboundary, then this inequality is related to the positivity of a co-dimension 1 Cheeger constant. 
In \Cref{sec: cheeger} we define the co-dimension 1 Cheeger constant, and state \Cref{thm: positive 1-cheeger constant} about its positivity.

In \Cref{sec: locally minimizing}, we show that for a suitable $\Phi$ most of $\calF$ is ``locally minimizing'', in the sense that the intersections with large (but fixed radius) balls are minimizing in these balls. 

Then, in \Cref{sec: tautness}, we show that a minimal (and locally minimal) 1-cochain is ``taut'' -- i.e.\ it is contained in a  (uniform) neighborhood of the the convex hull of its coboundary.

Finally, in \Cref{sec: thick}, we show that the convex hulls of minimal 1-cochains are ``thick'' (see \Cref{def: plump,def: barycenterically thick}), and use this to show that $|\lambda| \le |\bfd\lambda|$.

In \Cref{sec: pf of main theorem}, we prove the left inequality of \eqref{main theorem inequality} by using the uniform quasi-surjectivity \cite[Proposition 7.1]{lazarovich2025finite} of the map $\Phi$. This finishes the proof of \Cref{main theorem}.

In \Cref{sec: applications}, we provide proofs for the applications mentioned above.

\subsection*{Acknowledgements} We gratefully acknowledge the hospitality of Ashoka University, the Tata Institute of Fundamental Research, and the Newton Institute during visits at various stages of this collaboration.

\subsection*{Setup and notation}

\subsubsection*{The graph $X$ and metric notation}
Throughout the paper we fix a geometrically rigid residually finite hyperbolic group $G$.
By passing to a finite index subgroup if necessary we may assume that $G$ is torsion free. 

We fix a graph $X$ on which $G$ acts freely and cocompactly, to be chosen in \Cref{sec: singular patterns}.
We denote by $X^0$ its vertices, and by $d(\cdot,\cdot)$ the shortest path metric on $X^0$.

For $A,B\subseteq X^0$ and $y\in X^0$, we denote by $$d(y,A) = \min \{d(x,y)\;|\;x\in A\}$$ the distance of $y$ to $A$ and by $$d(A,B) = \min\{d(a,b)\;|\;a\in A,b\in B\}$$ the distance between the subsets $A,B$. We denote by $$\calN_R(A) = \{y\in X^0\;|\;d(y,A)\le R\}$$ the $R$-neighborhood of $A$.
The $R$-ball around $x$ is the set $\calB_R(x) = \calN_R(\{x\})$.

We denote by $[x,y]$ any geodesic path between $x,y$.

\subsubsection*{Delta notation}

We use $\bdelta$ to denote any constant that depends only\footnote{The constant $\bdelta$ will depend also on the fixed bicombing $q$ on $X$, see \Cref{the bicombing} and the remark following it} on $X$. 
For example, the hyperbolicity of $X$ can be written as: For all $x,y,z\in X^0$, $[x,y]\subseteq \calN_\bdelta([x,z]\cup[z,y])$.
The bold delta notation allows us to write sentences like ``the number of points in a ball of radius $\bdelta$ is $\bdelta$'', or equations like ``$\bdelta^2 = \bdelta = 2\bdelta$'', as each appearance of $\bdelta$ stands for a possibly different constant that depends only on $X$.
We will use $\bdelta(D,E,F,\dots)$ to denote a constant that depends on $X$ and on the parameters $D,E,F,\dots$.

\subsubsection*{The complex $Y$}

For $D\ge 1$, the \emph{Rips complex} $\Rips_D(X)$ is the simplicial complex with vertices $X^0$, and simplices spanned by $v_1,\dots,v_n$ whenever the diameter $\diam\{v_1,\dots ,v_n\}\le D$.
Throughout the paper, let $Y$ be the 2-skeleton of $\Rips_D(X)$ for some $D$.
Note that the dependence of $Y$ on $D$ is suppressed in our notation.
We endow the vertices $Y^0=X^0$ with the metric $d(\cdot,\cdot)$ previously defined on $X^0$. Note that this metric is not the same as the path metric on the 1-skeleton $Y^1$ of $Y$.\footnote{It is however quasi-isometric to it with quasi-isometry constants that depend on $D$.}

Note that if $H\le G$ is a finite index subgroup, then 
        \begin{equation}\label{eq: vol < index}
            \vol(Y/H)\le \bdelta(D) [G:H]
        \end{equation}
where $\vol(Y/H)$ denotes the number of simplices of all dimensions in $Y/H$.

We note that $Y/H$ is not necessarily a simplicial complex, but there exists $\dot G\le G$ (depending on $D$) such that if $H\le \dot G$ then $Y/H$ is simplicial. 

\subsubsection*{Homology and Cohomology}
For a simplicial complex $Z$, let $C_i(Z)$ be the simplicial $i$-chains with $\bbZ/2\bbZ$ coefficients and let $\partial:C_i(Z)\to C_{i-1}(Z)$ be the boundary map.
We identify elements of $C_i(Z)$ with finite sets of simplices of $Z$. 

Let $C^i_c(Z)$ be the compactly supported $i$-cochains with $\bbZ/2\bbZ$ coefficients, and let $\bfd:C^i_c(Z)\to C^{i+1}_c(Z)$ be the coboundary map.
We again identify elements of $C^i_c(Z)$ with finite sets of simplices of $Z$.

\begin{remark}
    Under these identifications $\alpha+\beta$ is the same as the symmetric difference of (the sets corresponding to) $\alpha$ and $\beta$.
We denote by $|\alpha|$ the size of $\alpha$ as a set.
\end{remark}

\section{Singular patterns}
\label{sec: singular patterns}

\subsection{Globally stable cylinders}
\label{subsec: globally stable cylinder}

It follows from the existence of globally stable cylinders \cite{petyt2025stable} by Petyt-Spriano-Zalloum  and the work of Rips-Sela \cite{rips1995canonical} that residually finite hyperbolic groups admit \emph{canonical representatives} in the following sense:
\begin{corollary}[Petyt-Spriano-Zalloum \cite{petyt2025stable},  Rips-Sela \cite{rips1995canonical}]\label{the bicombing}
 Let $G$ be a residually finite, torsion free, hyperbolic group, there exists a graph $X$ with a free and cocompact $G$-action, and a function $q:X^0\times X^0 \to C_1(X)$ such that for all $x,y,z\in X^0,g\in G$:
\begin{enumerate}[label = (Q\arabic*)]
    \item \label{quasi geodesic} $q(x,y) = e_1+\dots+e_n$ where $e_1,\dots,e_n$ is a $\bdelta$-quasi-geodesic simple edge-path from $x$ to $y$.
    \item $q(x,y) = q(y,x)$, 
    \item $gq(x,y) = q(gx,gy)$, and 
    \item \label{bounded defect} $|q(x,y)+q(y,z)+q(z,x)|\le \bdelta$.
\end{enumerate}
\end{corollary}

\begin{proof}
  Let $Z$ be the Cayley graph of $G$ with respect to some generating set.
  By \cite[Theorem 1.1]{petyt2025stable}, $G$ admits globally stable cylinders. This is a map $C:G\times G\to C_0(Z)$ such that for all $x,y,z\in Z^0$:
  \begin{enumerate}[label = (C\arabic*)]
      \item $[x,y]\subseteq C(x,y)\subseteq \calN_\bdelta([x,y])$
      \item $C(x,y) = C(y,x)$
      \item $gC(x,y)=C(gx,gy)$ for all $g\in G$, and
      \item $|C(x,y) +C(y,z)+C(z,x)|\le \bdelta$.
  \end{enumerate}
  Given such cylinders, Rips-Sela \cite{rips1995canonical} partition $C(x,y) = S_1\sqcup \dots \sqcup S_n$ into an (ordered) sequence of disjoint \emph{slices}.
  We denote by $\calS(x,y)$ the sequence $S_1,\dots,S_n$. We summarize their properties, as they appear in \cite{delzant1995image}. Let $x,y,z\in Z^0$:
  \begin{enumerate}[label = (S\arabic*)]\setcounter{enumi}{-1}
  \item $\diam(S)\le \bdelta$ for every slice $S$ of $C(x,y)$ (by \cite[Lemme I.1 (a)]{delzant1995image}).
  \item  If $\calS(x,y)=(S_1,\dots,S_n)$ then $x\in  S_1 ,y\in S_n$ and for all $1\le i,j\le n$ we have $$\tfrac 1\bdelta|i-j|-\bdelta\le  d(S_i,S_j)\le \bdelta|i-j|+\bdelta.$$
  The inequality on the left is from \cite[Lemme I.1 (a)]{delzant1995image}, while the inequality on the right is a consequence of the triangle inequality and \cite[Lemme I.1 (b)]{delzant1995image}.
  \item $\calS(x,y) = \calS(y,x)$ with the order reversed \cite[Lemma 3.4]{rips1995canonical}.
  \item $g\calS(x,y) = \calS(gx,gy)$ for all $g\in G$, and 
  \item The slices of $|\calS(x,y)+\calS(x,z)+\calS(y,z)|\le \bdelta$, 
  where $+$ denotes the symmetric difference.
This follows from (C4) and \cite[Lemme I.1 (c)]{delzant1995image}.
  \end{enumerate}
  
    Let $X$ be the graph whose vertices are non-empty subsets of $Z$ of diameter at most $\bdelta$, and two vertices $S_1,S_2$ are connected by an edge if $d_Z(S_1,S_2)\le \bdelta$.
    Since $G$ is torsion free, $G$ acts on $X$ freely and cocompactly.

    Let $p:X^0\to Z^0$ be a $G$-equivariant choice function, i.e. $p(x)\in x$.
    Define $q(x,y)$ to be the path in $X$ whose vertices are $x, S_1,\dots,S_n,y$, where $S_1,\dots,S_n$ are the slices of $C(p(x),p(y))$.\footnote{Here it should be understood that if $x=S_1$ or $S_n=y$ we take the appropriate simple path without these repetitions.}
\end{proof}

The function $q$ in \Cref{the bicombing} is referred to as a \emph{bicombing} on $X$. Throughout the paper, we will fix a bicombing $q$ as in \Cref{the bicombing}. The constants denoted by $\bdelta$ in fact depend on $X$ and the fixed bicombing $q$.

\


\subsection{The singular pattern}

Let $H\le G$ be a finite index subgroup. 
 Recall that $Y$ is the 2-skeleton of the Rips complex $\Rips_D(X)$ of $X$ for some $D$.
Let $\phi:H\to G$ be an injective homomorphism, and let $\Phi:Y^0 \to X^0$ be a $\phi$-equivariant map.

The following construction of a singular pattern is inspired by Delzant's singular foliation \cite{delzant1995image}:

\textbf{Regular connectors:} For each edge $e$ of $Y$ with endpoints $u,v$, map the edge $e$ maps linearly to the edge path $q(e) = q(\Phi(u),\Phi(v))$.
The preimages of midpoints of edges in $q(e)$ are points in the interior of $e$ which we call \emph{regular connectors}. If $f\in q(e)$ we denote its corresponding connector on $e$ by $\bfc_{e,f}$. See \Cref{fig: singular pattern}.

\textbf{Regular Segments:} 
Let $\Delta$ be a 2-simplex in $Y$ with edges $e_1,e_2,e_3$.
If an edge $f$ of $X$ is in $q(e_1)\cup q(e_2) \cup q(e_3)$ but not in $q(e_1)+q(e_2)+q(e_3)$, then $f$ appears in exactly two out of the three edge paths, say $q(e_1),q(e_2)$. We connect the two regular connectors $\bfc_{e_1,f}$ and $\bfc_{e_2,f}$ with a straight line segment in $\Delta$. These are called \emph{regular segments}.


\textbf{Singular segments and dead-ends:} 
Let $\Delta,e_1,e_2,e_3$ be as above, and let $f \in q(e_1)+q(e_2)+q(e_3)$. Then $f$ belongs to either one or three of $q(e_1),q(e_2),q(e_3)$. 
If $f\in q(e_i)$, we introduce a new point $\bfc_{e_i,\Delta,f}$, near $\bfc_{e_i,f}$, in the interior of $\Delta$, and connect it with a short line segment. We call the point $\bfc_{e_i,\Delta,f}$ a \emph{dead-end connector} and the line segment a \emph{singular segment}. 
We choose these points so that the singular line segments are disjoint. See \cref{fig: singular pattern}.


\begin{figure}
    \centering
    \includegraphics[width=\textwidth]{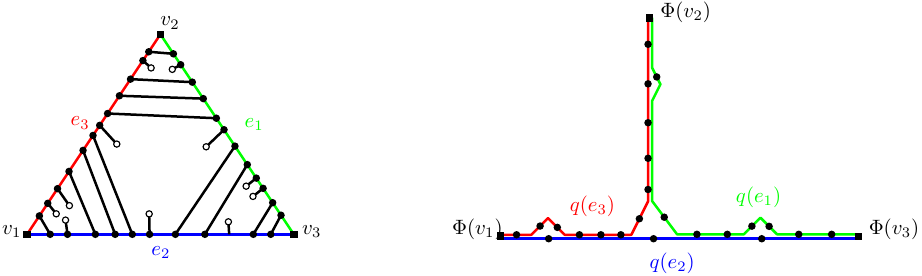}
    \caption{On the left, the singular pattern $\calF$ on a 2-simplex, and on the right the (quasigeodesic) triangle with sides given by edge-paths $q(e_1),q(e_2),q(e_3)$ and vertices $\Phi(v_1),\Phi(v_2),\Phi(v_3)$.}
    \label{fig: singular pattern}
\end{figure}

The graph $\calF$ whose vertices are connectors and edges are segments in $Y$ is called the \emph{singular pattern associated with $\Phi$ (and $q$)}.

\begin{remark} \label{tracks might intersect}
Note that regular segments are allowed to cross each other: Indeed, if two edges, say $f,f'$, appear in different orders in the quasigeodesic paths $q(e_1),q(e_2)$, then the regular segment between $\bfc_{e_1,f},\bfc_{e_2,f}$ intersects the one between $\bfc_{e_1,f'},\bfc_{e_2,f'}$.
In \cite{delzant1995image}, each intersecting regular segment is replaced by two non-intersecting singular segments. We do not do this in order to get a correspondence with compactly supported 1-cochains, as explained in \cref{section: tracks cochains}.  
Though we will not use this fact, we remark that by \cite[Lemme I.1(e)]{delzant1995image} the number of intersections in a 2-simplex is bounded. 
\end{remark}

The singular pattern $\calF$ is $H$-invariant, and so it gives rise to a singular pattern $\overline\calF$ on $\overline Y = Y/H$.
The graph $\overline\calF$ is a finite graph on the compact space $\overline Y$, and so it makes sense to define its size $|\overline\calF|$ to be the total number of regular connectors in $\overline\calF$.

Now, let $\Phi:Y^1\to X^1$ be a $\phi$-equivariant extension of $\Phi:Y^0\to X^0$ to the edges of $Y$. 
Let $e_1,\dots,e_t$ be representatives of $H$-orbits of edges in $Y$ (note that $t\le \bdelta(D)\cdot [G:H]$), define $$\TD_Y(\Phi) := \sum_{i=1}^t \ell(\Phi(e_i)) $$
If we assume further that $\Phi:Y^1\to X^1$ maps each edge to a geodesic path (e.g. if $\Phi$ minimizes $\TD_Y(\Phi)$),  then $$\TD_Y(\Phi)  =\sum_{i=1}^td(\Phi(\partial e_i))$$ 
where $d(\Phi(\partial e_i))$ is the distance between the images of the two endpoints $\partial e_i$ of $e_i$.


Every edge of $X$ is an edge of $Y$, so the edges $e_1,\dots,e_t$ contain $H$-orbit representatives of the edges of $X$. 
Hence,
$$ \TD(\Phi|_X)\le \TD_Y(\Phi).$$
The number of connectors of $\calF$ on an edge $e$ of $Y$ is the size of $q(e) = q(\Phi(\partial e))$ which by property \ref{quasi geodesic} can be bounded below by $d(\Phi(\partial e))$. 
Thus, if $\Phi:Y\to X$ maps edges to geodesics, then
$$ \TD_Y(\Phi) \le|\overline\calF|. $$

Combining these two inequalities we get 
\begin{equation}\label{eq: AD < F} 
\TD(\Phi)\le |\overline\calF|.
\end{equation}

        By property \ref{bounded defect}, the number of dead-ends in each 2-simplex $\Delta$ is at most $\bdelta$, and so the total number of dead-ends in $\overline\calF$ is at most $\bdelta \vol(\overline Y)$, where \(\vol (\overline Y)\) denotes the number of 2-simplices in \(\overline Y\).
        Thus, 
        \begin{equation}\label{eq: deadends < vol}
            \#\{\text{dead-ends in }\overline\calF\}\le \bdelta \vol(\overline Y).
        \end{equation}

\subsection{Tracks and compactly supported 1-cochains}\label{section: tracks cochains}

The connected components of the singular pattern $\calF$ (or $\overline\calF$) are called \emph{tracks}.
Denote by $|\lambda|$ the number of regular connectors in a track $\lambda$, then clearly
        \begin{equation}
        \label{eq: F is sum of tracks}
            |\overline\calF| = \sum_{\lambda\subseteq \overline\calF\text{ track}} |\lambda|
        \end{equation} 
        and
        \begin{equation}\label{eq: dead-ends of tracks and F}
        \sum_{\lambda \subseteq \overline\calF\text{ track}}\#\{\text{dead-ends of }\lambda\} = \#\{\text{dead-ends of }\overline\calF\}.
        \end{equation}

        Each track $\lambda$ of $\calF$ corresponds to a single edge $f$ in $X$, in particular it is embedded in $Y$ and its set-wise stabilizer in (the torsion-free group) $H$ is trivial. It follows that $\lambda$ embeds in $\overline Y$ under the quotient $Y\to \overline Y$.
        In particular, each track of $\calF$ is compact. 


        By property \ref{quasi geodesic}, each track $\lambda$ in $\calF$ intersects each edge of $Y$ at most once.
        Thus the track $\lambda$ gives rise to the element in $C_c^1(Y)$ which is the set of all edges $\lambda$ meets. By abuse of notation we denote $\lambda \in C_c^1(Y)$ as well.
        This abuse of notation is justified since the track $\lambda$ can be reconstructed (up to a small isotopy that does not pass through vertices of $Y$) from the corresponding 1-cochain in $C_c^1(Y)$. Indeed, there is a regular connector of the track on each edge of the 1-cochain: for a 2-simplex $\Delta$ we connect two regular connectors by a regular segment in $\Delta$ if they are the only two connectors on $\partial \Delta$;  otherwise the regular connectors on $\partial \Delta$ are connected to singular connectors in the interior of $\Delta$.
        Since the number of regular connectors in $\lambda$ is the size of the support of the corresponding 1-cochain, there is no ambiguity in the notation $|\lambda|$.
        
        It is easy to see that the coboundary $\bfd \lambda$ of a track $\lambda\subset \calF$ is supported in the set of 2-simplices that contain a dead-end of the track $\lambda$. In particular,
        \begin{equation}\label{eq: coboundary and deadends}
            |\bfd \lambda|\le \#\{\text{dead-ends of }\lambda\}
        \end{equation}
        At first glance one might suspect that \eqref{eq: coboundary and deadends} should be an equality, however $|\bfd\lambda|$ could be strictly smaller than $\#\{\text{dead-ends of }\lambda\}$. This happens when $\lambda$ meets all three of the sides of a 2-simplex $\Delta$, as there are three dead-ends of $\lambda$ in $\Delta$ but $\Delta$ is only counted once in $|\bfd\lambda|$.

        A word of caution is in order: The last two paragraphs concern only tracks in $\calF$ and not $\overline\calF$. A track in $\overline\calF$ might intersect the same edge of $\overline Y$ multiple times and so does not correspond as nicely to a 1-cochain on $\overline Y$.
        To remedy that, for a track $\overline\lambda$ in $\overline\calF$ we will interpret $|\bfd\overline\lambda|$ as $|\bfd \lambda|$ for some (any) lift $ \lambda$ of $\overline\lambda$ to $\calF$. 
        By doing so, \eqref{eq: coboundary and deadends} remains true also for tracks in $\overline\calF$.

\bigskip

\subsection*{Summary of \Cref{sec: singular patterns}}
Combining \eqref{eq: AD < F},\eqref{eq: F is sum of tracks}, and 
\eqref{eq: coboundary and deadends}, \eqref{eq: dead-ends of tracks and F}, \eqref{eq: deadends < vol}, \eqref{eq: vol < index} we get the following.
\begin{proposition}\label{prop: two halfs of the main inequality}
For $D\ge 0$, the singular pattern $\overline \calF$ defined above satisfies
\begin{equation}\label{eq: two halfs of the main inequality}
    \TD(\Phi)\le  \sum_{\overline\lambda \subseteq \overline\calF}|\overline\lambda| \quad \text {and} \quad \sum_{\overline\lambda \subseteq \overline\calF} |\bfd \overline\lambda| \le \bdelta(D) [G:H]. \qedhere
\end{equation}    
\end{proposition}

 \bigskip 
 
To prove inequality \eqref{the right inequality} of \Cref{main theorem} it remains to show that
$$\sum_{\overline\lambda \subseteq \overline\calF}|\overline\lambda| \le  \bdelta \sum_{\overline\lambda \subseteq \overline\calF} |\bfd \overline\lambda|.$$
One way of achieving this inequality is to prove that for each track $\lambda\subset \calF$ we have 
\begin{equation}\label{eq: codim 1 cheeger}
|\lambda|\le \bdelta |\bfd \lambda|.
\end{equation}
While we do not prove exactly this for all tracks, the gist of the argument is to show that this inequality holds  for \textbf{most} of $\overline\calF$.

We discuss \eqref{eq: codim 1 cheeger} and its relation to the classical Cheeger constant in the next section. 
As we said, we do not prove that \eqref{eq: codim 1 cheeger} holds for all tracks but only for most of them. In fact, some tracks might have no dead-ends, and so $|\bfd \lambda|=0$. 
We call such tracks \emph{regular}.
Equivalently, a track is regular if all its connectors are regular. It is called \emph{singular} otherwise.

The following proposition takes care of bounding regular tracks. Its proof is contained in the proof of \cite[Proposition 6.1]{lazarovich2025finite}.
\begin{proposition}\label{prop: bound on regular tracks}
    If $Y$ is one-ended and $\Phi$ minimizes $\TD_Y(\Phi)$ then  
    \begin{equation}\label{eq: regular < vol}
       \sum_{\lambda\subset \overline\calF \text{ is regular}}|\lambda| \le \bdelta \vol(\overline Y)
    \end{equation}
\end{proposition}

We will not make direct use of \Cref{prop: bound on regular tracks} in our proof of \Cref{main theorem}. 
However, we find it useful to mention it here, as it helps shift the focus to the task at hand -- bounding the size of the singular tracks.
In addition, as each track is either regular or contains a dead-end, one can bound the number of tracks by combining the bound on regular tracks \eqref{eq: regular < vol} and the bound on dead-ends \eqref{eq: deadends < vol}.

\begin{corollary}[{\cite[Proposition 6.1]{lazarovich2025finite}}]\label{bound on number of tracks}
    If $\Phi$ minimizes $\TD(\Phi)$ then the number of tracks in $\calF$ is at most $ \bdelta \vol(\overline Y)$.
\end{corollary}

\section{Connection to higher dimensional Cheeger constants}
\label{sec: cheeger}
Let $Z$ be a simplicial complex. 
Every 0-cochain $\nu\in C^0_c(Z)$ can be viewed as a finite set of vertices in $Z$. Its co-boundary $\bfd\nu$ is the set of edges connecting $\nu$ and its complement $Z^0-\nu$.
For an infinite $Z$, we define its Cheeger constant by
\begin{equation}\label{eq: 0-Cheeger}
h(Z) := \inf_{\nu\in C^0_c(Z)} \frac {|\bfd\nu|}{|\nu|}.
\end{equation}
If a group acts geometrically on a locally finite graph, then its Cheeger constant vanishes if and only if the group is amenable. Non-elementary hyperbolic groups are non-amenable (see \cite{gromov1987hyperbolic}), and so $h(Z)>0$. 
It follows that there exists a constant $c=c(Z)$ for all $\nu\in C^0_c(Z)$ we have $|\nu|\le c |\bfd \nu|$.

The desired inequality \eqref{eq: codim 1 cheeger} alludes to a co-dimension-1 analogue of the co-dimension-0 Cheeger constant $h(Z)$.

\begin{definition}[coboundary Cheeger constant]
 Let $B^{i+1}_c(Z)\subseteq C^{i+1}_c(Z)$ be the set of (compactly supported) $(i+1)$-coboundaries.
For  $\alpha\in B^{i+1}_c(Z)$  let $\MP(\alpha)$ be the size of the minimal primitive of $\alpha$ i.e.
$$\MP(\alpha) := \min_{\lambda\in C_c^i(Z): \alpha = \bfd \lambda}|\lambda|$$
and define the \emph{$i$-th coboundary Cheeger constant} to be
\begin{equation}\label{eq: codim i cheeger}
     h^i(Z) := \inf_{0\ne \alpha\in B^{i+1}_c(Z)}\frac {|\alpha|}{\MP(\alpha)}
\end{equation}
\end{definition}
\begin{remark}
    We have $h(Z) = h^0(Z)$ since if $Z$ is infinite then $\nu\in C^0_c(Z)$ is the unique (finitely supported) primitive of $\bfd\nu$.
\end{remark}

In \cref{sec: theorem b}, 
we will prove that rigid hyperbolic groups have positive co-dimension-1 Cheeger constant:
\begin{restatable}{theoremA}{cheegertheorem}\label{thm: positive 1-cheeger constant}
    If $G$ is a geometrically rigid hyperbolic group, and $G$ acts freely cocompactly on a simply connected simplicial complex $Z$, then $h^1(Z)>0$.
\end{restatable}

\begin{remark}
    It is worth mentioning as an aside how \cref{thm: positive 1-cheeger constant} differs from a somewhat related (co)homological characterization of hyperbolicity. Gersten \cite{gersten} showed that vanishing of the second $l^\infty$ cohomology for a finitely presented group characterizes hyperbolicity. Further, Allcock and Gersten \cite{allcock-gersten} (in particular the proof of Theorem 4.5 in \cite[p. 731]{allcock-gersten}) showed that a finitely presented group $G$ is (word) hyperbolic if and only if  the first unreduced $l^1$ homology group and the second reduced $l^1$ homology group of $G$ with real coefficients both vanish.  However, in \cref{thm: positive 1-cheeger constant} above, we are using $\bbZ/2\bbZ$ coefficients, and as the remark below shows, $h^1(Z)>0$ \emph{does not} characterize hyperbolicity.
\end{remark}

\begin{remark} \Cref{thm: positive 1-cheeger constant} is false when $G$ splits over $\bbZ$: To see this, consider $\alpha\in C^2_c(Z)$ with $|\alpha|=2$, i.e. $\alpha$ is supported on two simplices, say $\Delta_1,\Delta_2$. If $\alpha$ happens to be a co-boundary then for any primitive $\lambda\in C^1_c(Z)$ with $\alpha = \bfd \lambda$, necessarily, $|\lambda|$ is at least the distance $d(\Delta_1,\Delta_2)$. If $G$ splits over $\bbZ$ then we can find a complex $Z$ which splits over a triangulated strip $[0,1]\times \bbR$. Any two simplices $\Delta_1,\Delta_2$ in the strip give rise to a co-boundary $\alpha = \Delta_1+\Delta_2\in C^2_c(Z)$. So $\frac{|\alpha|}{\MP(\alpha)}\le \frac 2{d(\Delta_1,\Delta_2)}$. By taking $\Delta_1,\Delta_2$ to be arbitrarily far simplices in the strip we get $h^1(Z)=0$.
\end{remark}

Note that \Cref{thm: positive 1-cheeger constant} does not prove that inequality \eqref{eq: codim 1 cheeger} holds for all $\lambda\in C^1_c(Z)$, but rather only for minimal ones:

\begin{definition}
    A 1-cochain $\lambda\in C^1_c(Z)$ is \emph{minimal} (or a \emph{minimizer}) if $|\lambda| = \MP(\bfd\lambda)$, i.e.\ if it is the minimal primitive of its coboundary.
\end{definition}

Using this definition we can reformulate \Cref{thm: positive 1-cheeger constant} as stating that: for $G$ and $Z$ as in the theorem there exists $c=c(Z)$ such that if $\lambda\in C^1_c(Z)$ is minimal then $|\lambda|\le c |\bfd\lambda|.$

\section{Locally minimizing tracks}
\label{sec: locally minimizing}
The goal of this section is to prove that if $\Phi$ minimizes $\TD(\Phi)$ then most of $\overline\calF$ is ``locally minimizing''. 

\subsection{Simple subtracks and the injectivity radius}

\begin{definition}
    A \emph{subtrack} is a connected induced subgraph of $\calF$ or $\overline\calF$.
\end{definition}
    Let $d^*(\cdot,\cdot)$ be the intrinsic path metric on $\overline\calF$ (or $\calF$) as a graph. 
    Note that the $d^*$ distance between connectors in distinct tracks of $\overline\calF$ is $\infty$.
    For a connector $\bfc$ and $R\ge 0$, we denote by $\calB^*_R(\bfc)$ the ball of radius $R$ in $\overline\calF$. 
    This ball is a subtrack of $\calF$, as it is necessarily connected.

    If $e,e'$ are the edges of $Y$ containing the connectors $\bfc,\bfc'$ of $\calF$ respectively, then 
    \begin{equation}\label{eq: d d* relation}
    d(e,e')\le D\cdot  d^*(\bfc,\bfc'), \text{ and }\calB^*_R(\bfc)\subseteq \calN_{D\cdot R}(e).
    \end{equation}
    
\begin{definition}
    A subtrack is \emph{simple} if it meets each edge at most once.
\end{definition}

    In $\calF$, all (sub)tracks are simple, but this is not necessarily the case for $\overline\calF$. 
    The next lemma shows 
    that they are simple for scales which are smaller than the injectivity radius.
    
\begin{definition}
    Let $\bar x$ be a vertex of $\overline Y$, the \emph{injectivity radius $\injrad(\overline Y,\bar x)$ of $\overline Y$ at $\bar x$} is the maximal $R$  such that the quotient map $Y\to\overline Y$ is injective on $\calB(x,R)$ for any lift $x\in Y^0$ of $\bar x$. 
    The \emph{injectivity radius $\injrad(\overline Y)$ of $\overline Y$} is given by $\min_{\bar x\in \overline Y^0}\injrad(\overline Y,\bar x)$.
\end{definition}

\begin{lemma}\label{ball subtracks are simple below injrad}
    Any subtrack of $\overline\calF$ contained in $\calB(\bar x,\injrad(\overline Y,\bar x))$ is simple. 
    In particular, if $\injrad(\overline Y)\ge \bdelta(D,R)$ then for every connector $\bfc$ in $\overline\calF$, the subtrack $\calB_*(\bfc,R)$ is simple.
\end{lemma}

\begin{proof}
    Let $\rho = \injrad(\bar x,\overline Y)$, and let $\overline\lambda$ be a subtrack contained in $\calB_\rho(\bar x)$. Let $x$ be a lift of $\bar x$ to $Y$. Since any subtrack of $\calF$ is simple, the lift $\lambda$ of the subtrack $\overline\lambda$ to $\calB_\rho(x)$ is simple. 
    The quotient map is injective on $\calB_\rho(x)$ so $\overline\lambda$ is also simple in $\overline Y$.
    
    By \eqref{eq: d d* relation}, the subtrack $\calB^*_R(\bfc)$ is contained in the ball $\calB_{D\cdot R+D}(x)$ for the vertex $x$ incident on the edge $e$ containing $\bfc$. This implies the second part of the lemma when $\injrad(\overline Y) \ge D\cdot R+D =\bdelta(D,R)$.
\end{proof}

\subsection{The congruence locus}

\begin{definition}
    A subtrack is \emph{regular} if it consists only of regular connectors and segments.
    Two regular simple subtracks $\lambda_0,\lambda_1$ in $\overline\calF$ are \emph{congruent} if they meet the same edges of $\overline Y$.
    A regular simple subtrack $\lambda'$ is \emph{$N$-congruent} if $\overline\calF$ contains $N$ distinct regular simple subtracks congruent to $\lambda'$.
\end{definition}

\begin{definition}
Let $R\ge 0$ and $N\in \bbN$, let $\bfc$ be a connector of a track $\overline\calF$.  Then $\overline\calF$ is \emph{$R$-locally $N$-congruent at $\bfc$} if the ball $B^*_R(\bfc)$ is $N$-congruent.

The \emph{$R$-locally $N$-congruence locus} $\overline\calF_{R,N}$ is the induced subgraph of $\overline\calF$ on the set of $R$-locally $N$-congruent connectors.
\end{definition}

The next lemma says that ``most'' of $\overline\calF$ is contained in the $R$-locally $N$-congruence locus.

\begin{lemma}\label{stack locus is big}
    Let $N\in \bbN, D\ge 0$ and $R\ge 0$ and assume $\injrad(\overline Y) \ge \bdelta(D,R)$. Then 
    \begin{equation}\label{eq: size of congruent locus}
    |\overline\calF \setminus \overline\calF_{R,N}|\le \bdelta(D,R)\cdot N\cdot  \vol(\overline Y).
    \end{equation}
\end{lemma}

\begin{proof}
    If $\rho= \bdelta(D,R)$, then for every connector $\bar\bfc$ of $\overline\calF$, the subtrack $\calB^*_R(\bfc)$ is a subtrack in $\calB_\rho(\bar x)$ for some $\bar x \in \overline Y^0$.
    If we assume $\injrad(\overline Y)\ge \rho$ then by \Cref{ball subtracks are simple below injrad}, $\calB^*_R(\bar \bfc)$ is simple for every connector $\bar \bfc$ in $\overline\calF$.

    A connector $\bar \bfc$ belongs to $ \overline\calF \setminus \overline\calF_{R,N}$ if $\calB^*_R(\bar\bfc)$ is either not regular or not $N$-congruent. We give a bound in each case.

    The number of dead-ends in $\calB_\rho(\bar x)$ is bounded by $\bdelta(D,\rho)$ since the number of dead-ends in a 2-simplex is bounded by $\bdelta$. 
    Thus, the number of subtracks in $\calB_\rho(\bar x)$ that are not regular is bounded by $\bdelta(D,\rho)$.
    It follows that the number of $\bar\bfc$ for which $\calB^*_R(\bar \bfc)$ is not regular is bounded by $\bdelta(D,R)\vol(\overline Y)$.

    The number $P$ of congruence classes of simple subtracks in $\calB_\rho(\bar x)$ is bounded by $\bdelta(D,\rho)$.
    Therefore, at most $\bdelta (D,\rho)\cdot N$ of them are not $N$-congruent.
    It follows that the number of connectors $\bar \bfc$ for which $\calB^*_R(\bar \bfc)$ is regular but not $N$-congruent is at most $ \bdelta(D,\rho)\cdot N \cdot \vol(\overline Y)$.
\end{proof}

\subsection{The congruence locus is locally minimizing}

\begin{lemma}\label{paths sharing edges}
    For all $\kappa\ge 0$, and $N\ge \bdelta(\kappa)$, the following holds.
    Let $\gamma$ (resp. $\gamma'$) be a simple $\kappa$-quasigeodesic path in $X$, and assume that it passes through the edges $f_1,\dots,f_N$ (resp. $f'_1,\dots,f'_N$) in that order. 
    If $\{f_1,\dots,f_N\} = \{f'_1,\dots,f'_N\}$ then either:
    \begin{enumerate}[label = (\roman*)]
        \item \label{option 1} $d(f_1,f_1')\le \bdelta(\kappa)$ and $d(f_N,f_N')\le \bdelta(\kappa)$, or
        \item $d(f_1,f_N')\le \bdelta(\kappa)$ and $d(f_N,f_1')\le \bdelta(\kappa)$.
    \end{enumerate}
    If moreover, $\gamma,\gamma'$ share an endpoint then \ref{option 1} holds. 
\end{lemma}

\begin{proof}
Consider the subpath $\eta$ of $\gamma$ between $f_1,f_N$, and the subpath $\eta'$ of $\gamma'$ between $f_1',f_N'$.
By our assumption $f_1,f_N$ are edges of $\eta'$, and so by the Morse lemma, $\eta$ is $\bdelta(\kappa)$ close to a subpath of $\eta'$.
Similarly, $\eta'$ is $\bdelta(\kappa)$ close to a subpath of $\eta$. 
The lemma follows.
\end{proof}

\begin{definition}
    Let $R\ge 0$. A connector $\bfc\in \calF$ is \emph{$R$-locally minimizing} if $\calB^*_R(\bfc)$ is minimal.
\end{definition}

The following is the key result of this section, it shows that if $\Phi$ minimizes $\TD_Y(\Phi)$ then the congruence locus is locally minimizing. 
Essentially the idea is that if a subtrack of the congruence locus is not minimal, then it cobounds a region with a minimizer of its coboundary (see \cref{fig:placeholder}). A local change to $\Phi$ in this region will reduce $\TD_Y(\Phi)$, contradicting minimality.

\begin{lemma}\label{lem: sufficiently congruent are minimal}
    Let $D,R\ge 0$. If $\injrad(\overline Y)\ge \bdelta(D,R)$ and $N\ge \bdelta(D,R)$, then for every connector $\bar\bfc\in \overline\calF_{R,N}^0$, if $\Phi$ minimizes $\TD_Y(\phi)$, then $\bar \bfc$ is $R$-locally minimizing.
\end{lemma}

\begin{proof}
We will show that any subtrack $\overline\lambda\subseteq \calB^*_R(\bar\bfc)$ is minimizing (and get in particular that $\overline\lambda= \calB^*_R(\bar\bfc)$ is minimizing). 
Consider its lift $\lambda$ to $Y$.
By the definition of the congruence locus, $\overline\lambda$ (and so $\lambda$) is a regular, simple, $N$-congruent subtrack.
Let $\lambda_1,\dots,\lambda_N$ be the $N$ congruent copies of $\lambda$.
See \Cref{fig:placeholder}
\begin{figure}
    \centering
    \includegraphics[width=\textwidth]{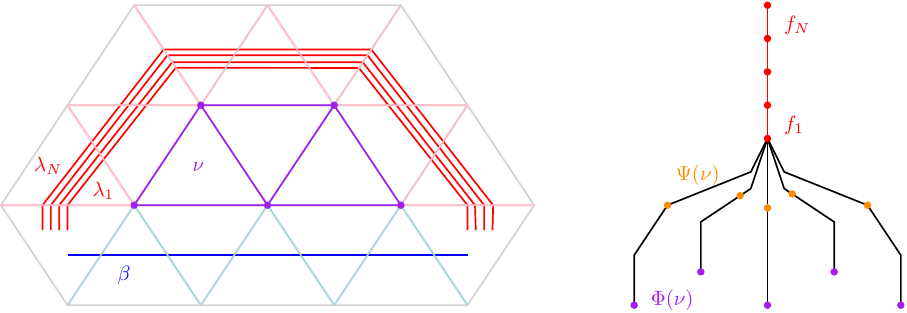}
    \caption{The complex $Y$ is on the left: the congruence tracks $\lambda_1,\dots,\lambda_N$ are shown in red and their edges in pink; a minimizing track $\beta$ is shown in blue and its edges in light blue; and the region $\nu$ is in purple.
    The space $X$ is on the right: The edges $f_1,\dots,f_N$ corresponding to $\lambda_1,\dots,\lambda_N$ are in red; the image of $\nu$ under $\Phi$, in purple, has to be far away from $f_i$, while its image under $\Psi$, in orange, is closer, and thus $\Psi$ has lower total distortion.}
    \label{fig:placeholder}
\end{figure}

Let $\beta$ be a minimizer with $\bfd \beta = \bfd \lambda$ and assume for the sake of contradiction that $|\beta|<|\lambda|$.
By replacing $\lambda$ by $\lambda - (\lambda\cap\beta)$ and $\beta$ by $\beta - (\lambda\cap \beta)$ we may assume that  $\lambda\cap \beta=\emptyset$.

The 1-cochain $\beta+\lambda$ is a 1-cocycle. Since $Y$ is one-ended $H^1_c(Y)=0$, and so there exists $\nu\in C^0_c(Y)$ such that $\bfd \nu=\beta+\lambda$.
For $\rho\ge \bdelta(D,R)$ the 0-cochain $\nu$ is supported in a ball $\calB(x,\rho)$ for some $x\in Y^0$.
Therefore, assuming $\injrad(\overline Y)>\rho$ we get that $\nu$ injects into $\overline Y$ 
under the quotient map, or in other words for all $h\in H$ we have $h \nu\cap \nu=\emptyset$.

Fix some edge $e_0$ in $\lambda$, and let $v_0,w_0$ be its endpoints such that $v_0\in \nu$ and $w_0\notin \nu$. 
Let $q(e_0) = q(\Phi(v_0),\Phi(w_0))$ be the $\bdelta$-quasigeodesic edge path given by the bicombing 
$q$, oriented from $\Phi(v_0)$ to $\Phi(w_0)$.

By construction of $\calF$ each $\lambda_i$ meets the edge $e_0$ in a connector corresponding to an edge $f_i$ in $q(e_0)$.
Up to reordering $\lambda_1,\dots,\lambda_N$ we may assume that $q(e_0)$ passes through $f_1,\dots,f_N$ in that order.

If $e'$ is any other edge in $\lambda$ then since $\lambda_1,\dots,\lambda_N$ are congruent, the edge path $q(e')$ passes through the same edges $f_1,\dots,f_N$. However, it might be that the order is changed. Re-enumerate $\{f'_1,\dots,f'_N\} = \{f_1,\dots,f_N\}$ such that $q(e')$ passes through $f'_1,\dots,f'_N$ in that order.

\begin{claim*} For $N\ge \bdelta$ if $e'$ is an edge of $\lambda$ and $f'_N$ is defined as above then 
\begin{equation}\label{eq: endpoint are close} d(f_N,f'_N)\le \bdelta\end{equation}
\end{claim*}
\begin{proof}[Proof of claim]
Since the subtrack $\lambda$ is connected, there is a sequence of edges $e_0,e_1,\dots,e_m=e'$ in $\lambda$ such that each consecutive pair shares an endpoint. For $0\le j \le m$ re-enumerate $\{f_1^j,\dots,f_N^j\} = \{f_1,\dots,f_N\}$ such that $q(e_j)$ passes through $f_1^j,\dots,f_N^j$ in that order. We will show by induction on $0\le j\le m$ that $d(f_N, f^j_N)\le\bdelta$. This proves the claim since $f'_N = f^m_N$.

The base case $j=0$ is clear since $f_N = f^0_N$. 
For the inductive step, assume that $d(f_N,f^j_N)\le \bdelta$. 
The $\bdelta$-quasigeodesic paths $q(e_j),q(e_{j+1})$ share an endpoint, so by \Cref{paths sharing edges}, we have $d(f_N^j,f_N^{j+1})\le \bdelta$. 
By \Cref{paths sharing edges} for the paths $q(e)$ and $q(e_{j+1})$, we have either $d(f_N,f_N^{j+1})\le \bdelta$ or $d(f_1,f_N^{j+1})\le \bdelta$. It suffices to show that the latter case is impossible. And indeed, since $q(e)$ is a $\bdelta$-quasigeodesic, we get a contradiction $$\tfrac 1 \bdelta N \le d(f_1,f_N) \le d(f_1,f_N^{j+1}) + d(f_N^{j+1},f_N^j)+1\le \bdelta$$ when $N\ge \bdelta$.
\end{proof}


Let $y \in X^0$ be an endpoint of the edge $f_N$.
Then by \ref{quasi geodesic}, $d(\Phi(v_0),y) \ge \tfrac 1\bdelta N$.

It follows from the claim that for every edge $e'$ in $\lambda$, if $v$ denotes the endpoint of $e'$ which is in $\nu$, then there exists $\kappa = \bdelta$ such that for $N\ge \bdelta$
\begin{equation}\label{eq: lower bound on length of edges of lambda'}
    d(\Phi(v),y) \ge d(\Phi(v),f'_N)-d(f_N',y)\ge \tfrac 1\bdelta N - \bdelta\ge \tfrac 1\kappa N.
\end{equation}

We will get a contradiction by constructing a $\phi$-equivariant map $\Psi:X\to X$ with $\TD_Y(\Psi)<\TD_Y(\Phi)$.
Define $\Psi:X\to X$ as follows:
Outside the $H$ translates of $\nu$, we set $\Psi|_{X-H\nu}=\Phi|_{X-H\nu}$. 
For each $v\in \nu$ we set $\Psi(v) =y$ if $d(v,y)<\tfrac 1\kappa N$, otherwise, we set $\Psi(v)$ to be the point on the geodesic $[\Phi(v),y]$ at distance $\tfrac 1\kappa N$ away from $\Phi(v)$.
We then extend $\Psi$ in a $\phi$-equivariant way to all of $X$, such that $\Psi$ maps edges to geodesic paths.

Let us compute $\TD_Y(\Psi)$. We do so by partitioning the possible orbits of edges into 4 sets of edges:

\noindent\textbf{Edges of $\lambda$:} 
Let $e$ be an edge of $\lambda$, with endpoints $v,w$ such that $v\in \nu$ and $w\notin \nu$. By definition $\Psi(w) = \Phi(w)$, and $\Psi(v)$ is obtained by shifting $\Phi(v)$ a distance $\tfrac 1\kappa N$ towards $y$. Therefore, 
\begin{align*}\ell(\Psi(e)) &= d(\Psi(v),\Psi(w))\\
&\le  d(\Psi(v),y)+d(y,\Phi(w))\\
& = d(\Phi(v),y) -\tfrac 1 \kappa N + d(y,\Phi(w))\\
&\le d(\Phi(v),\Phi(w))+\bdelta - \tfrac 1\kappa N\\
& =\ell(\Phi(e)) + \bdelta - \tfrac 1\kappa N
\end{align*}
where the second to last inequality follows from the claim since the point $y$ is at distance $\bdelta$ away from the quasigeodesic path $q(e)$ and so (by the Morse lemma) at distance $\bdelta$ away from the geodesic $[\Phi(v),\Phi(w)]$. 
\begin{equation}
\sum_{e\in\lambda} \ell(\Psi(e)) \le \sum_{e\in\lambda} \left(\ell(\Phi(e))-\tfrac 1\kappa N + \bdelta\right) = \left(\sum_{e\in\lambda} \ell(\Phi(e)\right) - |\lambda|\cdot \tfrac 1 \kappa N + \bdelta |\lambda|.
\end{equation}

\noindent\textbf{Edges of $\beta$:} Let $e$ be an edge of $\beta$, let $v\in \nu$ and $w\notin \nu$ be its endpoints. Then
\begin{align*}
\ell(\Psi(e)) &= d(\Psi(v),\Psi(w))\\
&\le  d(\Psi(v),\Phi(v))+d(\Phi(v),\Phi(w))\\
& \le \tfrac 1 \kappa N + \ell(\Phi(e)).
\end{align*}
Hence
\begin{equation}
\sum_{e\in\beta} \ell(\Psi(e)) \le \sum_{e\in\beta} \left(\ell(\Phi(e))+\tfrac 1\kappa N\right) = \left(\sum_{e\in\beta} \ell(\Phi(e))\right)+|\beta|\cdot \tfrac 1\kappa N
\end{equation}

\noindent\textbf{Edges in $\nu$:} If $e$ has both endpoints $v,w\in \nu$ then since the triangle $\Phi(v),\Phi(w),y$ is $\bdelta$-thin, we get
\begin{equation*}
    \ell(\Psi(e)) = d(\Psi(v),\Psi(w))\le d(\Phi(v),\Phi(w))+\bdelta
\end{equation*}
The number of such edges can be bounded 
$$\#\{\text{edges with both endpoints in }\nu\}\le \bdelta(D)|\nu|.$$ 
Using the Cheeger constant (see \eqref{eq: 0-Cheeger})  and minimality of $\beta$ (i.e. $|\beta|\le |\lambda|$) we have
$$|\nu|\le \bdelta(D) |\bfd \nu|\le \bdelta(D) (|\lambda|+|\beta|)\le \bdelta(D) |\lambda|.$$
Therefore,
\begin{equation}
\sum_{e\text{ in }\nu} \ell(\Psi(e)) \le \left(\sum_{e\text{ in }\nu}\ell(\Phi(e))\right)+\bdelta(D)|\nu|\le \left(\sum_{e\text{ in }\nu} \ell(\Phi(e))\right)+\bdelta(D)|\lambda|
\end{equation}

\noindent\textbf{Edges outside $\nu$: } Finally, for the edges whose endpoints are both in $X - H\nu$, we have $\ell(\Phi(e)) = \ell(\Psi(e))$.

\bigskip 

Combining the above we have,
\begin{align*}
    \TD_Y(\Psi)&\le \TD_Y(\Phi)-|\lambda|\cdot \tfrac 1\kappa N+|\beta|\cdot \tfrac 1\kappa N+|\lambda|\cdot \bdelta(D)\\
    &\le \TD_Y(\Phi) - \tfrac 1\kappa N +|\lambda|\cdot \bdelta(D)
.\end{align*}
Since $\lambda$ is in the ball $B^*_R(x)$ we have $|\lambda|\le \bdelta(D,R)$. By choosing $N\ge \bdelta(D,R)$ we have 
$$\TD_Y(\Psi)<\TD_Y(\Phi)$$
contradicting its minimality.
\end{proof}

\subsection*{Summary of \Cref{sec: locally minimizing}}

\begin{definition}
    The \emph{$R$-locally minimal locus} $\overline\calF_R$ is the induced subgraph on the set of all $R$-locally minimizing connectors in $\overline\calF$.
\end{definition}

Combining \Cref{ball subtracks are simple below injrad,stack locus is big,lem: sufficiently congruent are minimal} we get

\begin{proposition}\label{complement of the minimizing locus}
    Let $R,D\ge 0$. If $\injrad(\overline Y)\ge \bdelta(D,R)$, and $\Phi$  minimizes $\TD_Y(\phi)$, then for $N\ge \bdelta(D,R)$ the congruence locus $\overline\calF_{R,N}$ is contained in the minimizing locus $\overline\calF_{R}$ and 
    \begin{equation}\label{eq: complement of the minimizing locus}|\overline\calF\setminus \overline\calF_R|\le \bdelta(D,R)\cdot \vol(\overline Y).
    \end{equation}
\end{proposition}

Recall that our goal is to show that $|\overline\calF|\le \bdelta(D)\vol(\overline Y)$.
Under the assumptions of the previous proposition we see that it suffices to establish a bound of the form $|\calF_R|\le \bdelta(D,R)\vol(\overline Y)$ for some $R$.






\section{Tautness} 
\label{sec: tautness}
A 1-cochain is ``taut'' if it is contained in (a uniformly bounded neighborhood of) the convex hull of its coboundary. 
The goal of this section is to prove that if $\lambda$ is a local minimizer then it is taut.
This will be done in three steps: first showing that a taut primitive exists, then showing that any minimizer is taut and finally showing that local tautness implies global tautness.

\subsection{Convex hulls}

We begin by gathering some basic facts and definitions regarding convex sets and nearest point projections.

\begin{definition}
    For $A\subseteq X$, the \emph{convex hull} $\conv(A)$ of $A$ is the union of all geodesics between points in $A$. 
    We denote by $\conv_R(A) = \calN_R(\conv(A))$ the $R$-neighborhood of the convex hull of $A$.
    A set $C\subseteq X$ is \emph{$\kappa$-quasiconvex} if $\conv(C)\subseteq \calN_\kappa(C)$.

    For $C\subseteq X$ and $x\in X$, a point $x'\in C$ is a nearest point to $x$ if $d(x,x') = d(x,C)$.
    We denote by $\pi_C:X\to C$ any function that maps $x\in X$ to a nearest point in $C$ and refer to $\pi_C$ as a nearest point projection to $C$.

    For $x,y,z\in X$, let $\gen{x,y}_z = \tfrac 12(d(x,z)+d(y,z) - d(x,y))$ denote the \emph{Gromov product of $x,y$ with respect to $z$}.
\end{definition}

\begin{lemma}\label{facts about qc sets} $ $
\begin{enumerate}[label = (\arabic*)]
    \item If $C$ is $\kappa$-quasiconvex then $\calN_R(C)$ is $\bdelta(\kappa)$-quasiconvex for all $R$ \cite[Proposition 10.1.2]{coornaertdelzantpap}.
    \item \label{fact: convexity of convex hulls} For all $A\subseteq Y$ the set $\conv_R(A)$ is $\bdelta$-quasiconvex for all $R\ge 0$ \cite[Proposition 10.1.3]{coornaertdelzantpap}.
\item If $x',x''$ are nearest points to $x$ in a $\kappa$-quasiconvex subset $C$, then $d(x',x'')\le \bdelta(\kappa)$ \cite[Proposition 10.2.1]{coornaertdelzantpap}.
\item 
If $C$ is a finite $\kappa$-quasiconvex subset of $X$, then for every $x_n\in X$ such that $x_n\to \xi\in \partial X$ the set of limits points of the projections $\{\pi_C(x_n)\}$ is non-empty and has diameter at most $\bdelta(\kappa)$. 
We extend $\pi_C:\partial X\to C$ by setting $\pi_C(\xi)$ to be any limit point of $\{\pi_C(x_n)\}$ for $x_n\to \xi$.
\item \label{nearest point quadrilateral} Let $C$ be a finite $\kappa$-quasiconvex subset of $X$, then for all $x,y\in X\cup \partial X$ if $\gen{x,y}_{\pi_C(x)}\ge \bdelta(\kappa)$ then $d(\pi_C(x),\pi_C(y))\le \bdelta(\kappa)$. \hfill \qedsymbol
\end{enumerate}
\end{lemma}

Items (4) and (5) of  Lemma~\ref{facts about qc sets} follow from the fact that nearest point projection  $\pi_C: X \to C$ is coarsely well-defined \cite[Lemmas 3.1, 3.2]{mitra-trees}.
We first provide  a quick sketch for (5) assuming that $y \in X$ and refer the reader to the argument in \cite[Lemmas 3.1, 3.2]{mitra-trees} for further details.
Suppose
$\gen{x,y}_{\pi_C(x)}$ is large as in the hypothesis. Let $z$ be the nearest point projection of  $y$ onto $[x, \pi_C(x)]$. 
Then $[y,z]\cup [z, \pi_C(x)]$ is a $\bdelta-$quasigeodesic. 
Hence $d(\pi_C(x),\pi_C(y))\le \bdelta(\kappa)$. 

To extend to $\partial X$ and prove (4), let $x_n \to \xi$ be any sequence of points in $X$. Then $\gen{x_n,x_m}_{\pi_C(x_n)}\to \infty$, and the proof of (5) in the paragraph above establishes that $\pi_C(\xi)$
is coarsely well-defined.

\begin{lemma}\label{projection of connected is quasi connected}
    Let $C$ be a finite $\kappa$-quasiconvex subset of $X$. Let $\pi_C:\partial X\to C$ be a nearest point projection. If $\Gamma\subseteq \partial X$ is path connected, then $\calN_{\bdelta(\kappa)}(\pi_C(\Gamma))$ is connected.
\end{lemma}

\begin{proof}
    By \Cref{facts about qc sets}\ref{nearest point quadrilateral} it follows that for every $\xi\in X\cup \partial X$ there exists a neighborhood $U$ such that $d(\pi_C(\xi),\pi_C(\zeta))\le \bdelta(\kappa)$ for all $\zeta\in U$. 
    The lemma easily follows.
\end{proof}

\begin{lemma}\label{lem: pushing away curves}
    Let $G$ be a one-ended hyperbolic group. 
    For every $D\ge \bdelta$, $\kappa\ge 0$ and $R \ge \bdelta(D,\kappa)$ if $C\subseteq Y$ is a finite $\kappa$-quasiconvex subset, then $C'=\calN_R(C)$ has the \emph{pushing away curves property}:
    \begin{enumerate}[label = (PAC)]
        \item \label{pushing away property} For every closed loop $\gamma$ in $Y-C'$ and every $r\ge 0$, the curve $\gamma$ can be freely homotoped in $Y-C$ into $Y-\calN_{R+r}(C)$ (See \Cref{fig:pac}).
    \end{enumerate}
\end{lemma}

\begin{figure}
    \centering
    \includegraphics[]{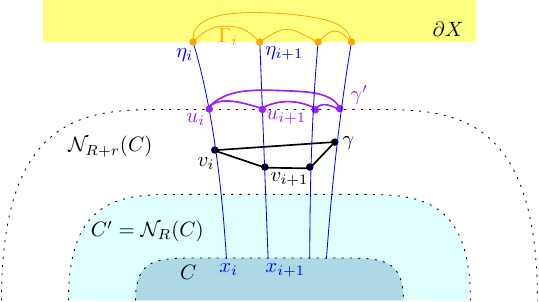}
    \caption{An illustration of the pushing away curves property \ref{pushing away property}, the outline of the free homotopy is shown in orange.}
    \label{fig:pac}
\end{figure}


    \begin{proof}
        Let $C' =\calN_R(C)$, and let $\gamma$ be a closed loop in $Y-C'$. 
        Without loss of generality we may assume that $\gamma$ is the combinatorial path $v_0,v_1,\dots,v_{m},v_0$. Since $v_i,v_{i+1}$ are adjacent in $Y$, we have $d(v_i,v_{i+1})\le D$.

        Let $x_i$ be a nearest point to $v_i$ in $C$.
        By \cite[Lemma 3.1]{bestvina1991boundary} there exists a geodesic ray $\eta_i$ starting at $x_i$, and a point $u_i\in \eta_i$ such that $d(v_i,u_i)\le \bdelta$.
        It follows that $d(u_i,u_{i+1})\le \bdelta(D)$ and $\gen{\eta_i,\eta_{i+1}}_{x_i}\ge R-\bdelta(D)$.

        It follows from \cite[Lemma 6.1 (1)]{lazarovich2023commensurated} that if $R\ge \bdelta(D)$ then there exists a path $\Gamma_i\subseteq\partial X$ between $\eta_i,\eta_{i+1}$ such that
        $(\xi,\xi')_{x_i}\ge R-\bdelta(D)$ for all $\xi,\xi'\in \Gamma_i$.

        Let $\rho\ge 0$, and consider $C'' = \calN_{R+\rho}(C)$. 
        By \Cref{facts about qc sets}\ref{fact: convexity of convex hulls}, $C''$ is $\bdelta(\kappa)$-quasiconvex.
        For $0\le i\le m$ let $u'_i$ be the point on $\eta_i$ such that $d(u'_i,x_i) = R+\rho$, in other words, $u'_i$ is a nearest point projection of $\eta_i$ to $C''$.
        By \Cref{projection of connected is quasi connected} the projection $\pi_{C''}(\Gamma_i)$ gives rise to a $\bdelta(\kappa)$-connected subset.
        In particular, we can find a path $\gamma'_i$ in $Y\ssm \calN_{R+\rho-\bdelta(\kappa)}(C)$  connecting $u'_i,u'_{i+1}.$
        Let $\gamma' = \gamma'_0 \gamma'_1\dots \gamma'_m$ be their concatenation.

        We claim that for $D\ge \bdelta$ and $R\ge \bdelta(D,\kappa)$ the loops $\gamma$ and $\gamma'$ are freely homotopic in $Y\ssm C$:

        For $0\le i\le m$ the closed loop $\omega_i = [v_i,u'_i]\;\gamma'_i\;[u'_{i+1},v_{i+1}]\;[v_{i+1},v_i]$
        is contained in the $\bdelta$-quasiconvex set $O_i = \{x\;|\;\gen{x,\eta_i}_{x_i}\ge R-\bdelta(D)\}.$ 
        For $D\ge \bdelta$ the Rips complex $\Rips_D(O_i)$ is contractible \cite[Ch.4]{GhH},
        and so there is a disk filling $\omega_i$ in $O_i$.

        Thus, $\gamma$ in $X\ssm \calN_R(C)$ can be freely homotoped in $X\ssm C$ to a loop $\gamma'$ in $X\ssm \calN_{R+\rho - \bdelta(\kappa)}(C)$ for all $\rho$. 
    \end{proof}

\subsection{Taut primitives}
As a first step, we show that every coboundary has a primitive supported in its convex hull.

\begin{definition}
    $\beta\in C^1_c(Y)$ is \emph{$T$-taut} if $\beta$ is supported in $\conv_T(\bfd \beta)$.
\end{definition}

\begin{lemma}\label{lem: primitive in convex hull}
    Let $\alpha\in B^2_c(Y)$. Then $\alpha$ has a $\bdelta$-taut primitive. That is, there exists $\beta\in C^1_c(Y)$ supported in $\conv_\bdelta(\alpha)$ such that $\alpha = \bfd \beta$. 
\end{lemma}

For $C\subseteq Y$, recall that the relative $i$-cochain group $C^i_c(Y,Y -C)$ is the set of all compactly supported $i$-cochains that are supported on $C$.

\begin{proof}
Let $C = \conv(\alpha)$. By \Cref{facts about qc sets}\ref{fact: convexity of convex hulls}, $C$ is $\bdelta$-quasiconvex. Now, by \Cref{lem: pushing away curves}, some neighborhood $C' = \conv_\bdelta(\alpha)$ satisfies \ref{pushing away property}. 

We know that $\alpha\in C^2_c(Y,Y-C)$, and $\alpha =0 \in H^2_c(Y)$ and we want to show that $\alpha = 0\in H^2_c(Y,Y-C')$. 
Since $C,C'$ are compact there exist long exact sequences in compactly supported cohomology for the pairs $(Y,Y-C)$ and $(Y,Y-C')$. By naturality we have the following commuting diagram:
\begin{equation}
    \begin{tikzcd}
   \dots  \arrow[r] & H^1_c(Y-C) \arrow[r,"\bfd"]\arrow[d,"r"] &H^2_c(Y,Y-C) \arrow[r,"\iota"]\arrow[d,"r"] &H^2_c(Y) \arrow[r]\arrow[d,"r"] &\dots\\
   \dots  \arrow[r] & H^1_c(Y-C') \arrow[r,"\bfd"] &H^2_c(Y,Y-C') \arrow[r,"\iota"] &H^2_c(Y) \arrow[r] &\dots
\end{tikzcd}
\end{equation}
Here the vertical arrows are the obvious restriction maps. 
Our goal is to prove that $r(\alpha)=0$.

By assumption, $\iota(\alpha)=0$, and so by exactness there exists $\beta\in H^1_c(Y-C)$ such that $\bfd(\beta)=\alpha$. By commutativity, 
$r(\alpha) = r(\bfd (\beta)) = \bfd (r (\beta)) \in H^2_c(Y,Y-C')$
and so the proof will be complete if we show that $r:H^1_c(Y-C) \to H^1_c(Y-C')$ is $0$.

Since $Y$ is one-ended, so are $Y-C, Y-C'$. Recall that a space has at most one end if and only if the comparison map $H^1_c \to H^1$ is injective. 
Thus the horizontal maps in the following commutative diagram are injective:
\begin{equation}
    \begin{tikzcd}
        H^1_c (Y-C) \arrow[hookrightarrow]{r}{c}\arrow[d,"r"] &H^1(Y-C) \arrow[d,"r"]\\
        H^1_c (Y-C') \arrow[hookrightarrow]{r}{c} &H^1(Y-C') 
    \end{tikzcd}
\end{equation}
Thus, it suffices to show that $r\circ c =0$.

Let $\beta \in C^1_c(Y-C)$. 
Let $\gamma$ be any closed curve in $Y-C'$.
By \ref{pushing away property},  the curve $\gamma$ can be homotoped in $Y-C$ to a curve which is arbitrarily far away from $C'$. In particular, it can be homotoped to a curve $\gamma'$ such that $\gamma'\cap \supp(\beta)=\emptyset$. 
We have $r(c(\beta)) (\gamma) = c(\beta)(\gamma) = c(\beta)(\gamma') = 0$. 
Since this holds for every closed curve $\gamma$ in $Y-C'$, by the Universal Coefficient Theorem, $r(c(\beta)) = 0 \in H^1(Y-C')$. Since this holds for every $\beta$, we get the desired $r\circ c =0$.
\end{proof}

\subsection{Minimizers are taut}

Our next goal is to show that the minimizers are taut.
In the proposition below, $D$ refers to the constant used in building the Rips' complex $Rips_D(Y)$.
\begin{proposition}\label{minimal is taut}
    For $D \ge \bdelta$, if $\lambda\in C^1_c(Y)$ is minimal then $\lambda$ is $\bdelta(D)$-taut.
\end{proposition}

We will need the following geometric lemma that tells us that far away from a quasiconvex set, there are more elements in a large ball that are further away from the set than closer to it. 
\begin{lemma}\label{lem: more far than close}
    For all $\kappa\ge 0$ if $C\subseteq X$ is $\kappa$-quasiconvex, $D>\bdelta(\kappa)$ and $R=d(x,C)\ge \bdelta(\kappa,D)$ then \begin{equation}\label{eq: more outside than in}
        |\calB_D(x)\cap \calN_{R}(C)|<\tfrac12|\calB_D(x)|
    \end{equation}
\end{lemma}

\begin{proof}
    Let $y\in \calB_D(x)$. Let $x',y' \in C$ be respectively  nearest points to $x,y$. 
    We have $\gen{x,y}_{x'}\ge R-D-\bdelta$. By \Cref{facts about qc sets}\ref{nearest point quadrilateral}, for $R\ge \bdelta(\kappa,D)$ we have $d(x',y')\le \bdelta(\kappa)$.



    
    Now, assume further that $d(y,C)\le d(x,C)=R$. 
    Then, $d(y,x')\le d(y,y')+d(y',x')\le  R+\bdelta(\kappa)$.

    Let $w$ be the point at distance $\tfrac D2$ away from $x$ on the geodesic $[x,x']$. See \Cref{fig:ball intersection nbd}.
    \begin{figure}
        \centering
        \includegraphics[height=0.2\textheight]{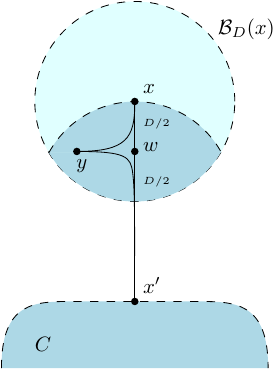}
        \caption{Intersection of a ball and a neighborhood of a quasiconvex set.}
        \label{fig:ball intersection nbd}
    \end{figure}
    We claim that $d(y,w)\le \tfrac D2+ \bdelta(\kappa)$:

    The triangle $xyx'$ is thin, and so $w$ is either $\bdelta$ close to the geodesic $[x,y]$ or to the geodesic $[y,x']$. In the former case, $$d(w,y)\le d(x,y)-d(x,w)+\bdelta = D/2+\bdelta.$$
    In the latter, $$d(w,y)\le d(y,x')-d(w,z)+\bdelta\le (R+\bdelta(\kappa)) - (R-\tfrac D2)+\bdelta=\tfrac D2 + \bdelta(\kappa).$$

    We have thus shown that 
    $$\calB_D(x)\cap \calN_R(C) \subseteq \calB_{D/2+\bdelta(\kappa)}(w)$$

    By \cite{coornaert1993mesures}, balls grow exponentially in $X$, and so
    for $D\ge \bdelta(\kappa)$ we have $$|\calB_{D/2+\bdelta(\kappa)}(w)|<\tfrac12 |\calB_D(x)|$$
    which together with the previous inclusion proves the lemma.
\end{proof}

\begin{proof}[Proof of \Cref{minimal is taut}]
Let $C = \conv(\alpha)$.
By \Cref{lem: primitive in convex hull}, there exists $\rho=\rho(X)$ and $\beta'\in C^1_c(Y)$ such that $\bfd \beta' = \alpha$ and $\beta'$ is supported in $\conv_{\rho}(\alpha)$.
Let $\beta \in C^1_c(Y)$ be a minimal primitive of $\alpha$.
Consider $\gamma = \beta+\beta'$. Clearly, $\bfd\gamma =0$. Since $Y$ is one-ended and $H^1(Y)=0$, we have  $H^1_c(Y)=0$.
Therefore, there exists $\nu\in C^0_c(Y)$ with $\bfd\nu = \gamma$.

Let $x\in \nu$ be the furthest point from $C$. 
Let $R=d(x,C)$. 
The element $\beta_1 = \beta+\bfd x$ is a (compactly supported) primitive for $\alpha$, since $\bfd \beta_1 = \bfd\beta=\alpha$.
We will show that if $D\ge \bdelta$ and $R\ge \bdelta(D)$, then $|\beta_1|<|\beta|$, contradicting minimality of $\beta$.

Consider the neighbors of $x$ in $Y$. By definition of the Rips complex, these are exactly the vertices $y\ne x$ in the ball $\calB_D(x)$.

By definition of $R$, $\nu$ is supported in $\calN_{R}(C)$. 
In particular, if the neighbor $y$ of $x$ is not in $\calN_{R}(C)$ then it is not in $\nu$, and the edge $e$ between $x$ and such $y$ belongs to $\bfd\nu = \gamma = \beta+\beta'$. 
Moreover, if $R > \rho$ then $e$ is necessarily an edge of $\beta$ (as it cannot be an edge of $\beta'$ which is contained in $\calN_\rho(C)$).
By \Cref{lem: more far than close}, if $D\ge \bdelta$ and $R>\max\{\bdelta(D),\rho\}=\bdelta(D)$ then most of the neighbors of $x$ are not in $\calN_{R}(C)$, and so $|\beta_1|=|\beta + \bfd x|<|\beta|$, in contradiction to the minimality of $\beta$.
\end{proof}

\subsection{Local minimizers are taut}

Recall that the $R$-locally minimal locus $\overline\calF_R$ is the induced graph on the connectors $\bar\bfc\in \overline\calF^0$ for which $\calB^*_R(\bar\bfc)$ is simple and minimal.

\begin{lemma}\label{geodesic outside convex set}
    Let $T_1\ge 0$, and $T_2,T_3\ge \bdelta(T_1)$. If $C=\conv(A)$ for some $A\subset X^0$ and if $w,x,y\in X$ satisfy 
    \begin{align}
     &d(w,[x,y])\le T_1, \\ 
     &\label{eq: w is the maximal} d(w,C) \ge \max \{d(x,C),d(y,C)\}\\
    &\label{eq: x options} x \in C \quad \text{ or }\quad d(w,x)\ge T_2\\
    &y \in C \quad \text{ or }\quad d(w,y)\ge T_2
    \end{align}
    then 
    $d(w,C)\le T_3$. 
\end{lemma}

\begin{proof}
Let $x,y,w$ be as in the lemma, and let $x',y'\in C$ be nearest points to $x,y$ respectively. 
Since $d(w,[x,y])\le T_1$ and the quadrilateral $x,y,x',y'$ is $\bdelta$-thin we see that one of the following holds:
\begin{enumerate}[label= (\alph*)]
    \item \label{w,x,x'}$d(w,[x,x'])\le T_1+\bdelta$,
    \item \label{w,x',y'}$d(w,[x',y'])\le T_1+\bdelta$ or
    \item \label{w,y,y'}$d(w,[y',y])\le T_1+\bdelta$.
\end{enumerate} 

In Case \ref{w,x,x'}, we get $$d(w,C)\le  d(x,x')-d(x,w)+2T_1+\bdelta.$$
There are two cases to consider, according to \eqref{eq: x options}: If $d(x,w)\ge T_2 \ge 2T_1+ \bdelta$ then we have $d(w,C) <d(x,x')$ contradicting \eqref{eq: w is the maximal}.
Otherwise, $x\in C$ so $x'=x\in C$ and we get the desired bound $d(w,C)\le 2T_1+\bdelta$.

Case \ref{w,y,y'} is done similarly.

In Case \ref{w,x',y'}, we get the desired bound by $\bdelta$-quasiconvexity of $C$:  $d(w,C)\le T_1+\bdelta$.
\end{proof}

\begin{lemma}\label{local minimizers are taut}
    For $D\ge \bdelta$, $R,T\ge \bdelta(D)$ any subtrack $\overline\lambda$ in $\overline\calF_R$ is $T$-taut. 
\end{lemma} 

\begin{proof}
Let $\overline\lambda\subseteq \overline\calF_R$, and let $\lambda\subseteq \calF_R$ be its lift.
Let $C = \conv(\bfd\lambda)$.
Let $\bfc\in \lambda$ maximize $d(\bfc,C)$. The goal is to bound $d(\bfc,C)$ by a constant $T$ depending only on $D,R$. 

By \Cref{minimal is taut}, for $D\ge \bdelta$ and $T_1= \bdelta(D)$ every minimal 1-cocycle is $T_1$-taut. 
Let $T_2=\bdelta(T_1)=\bdelta(D)$ be as in \Cref{geodesic outside convex set}.

Let $R_0=  \bdelta(T_2)=\bdelta(D)$ be the number of edges in the ball of radius $T_2$ in $Y$.
Let $R\ge R_0$.
Consider the ball $\beta = \calB^*_R(\bfc)$. We have the following two cases:

\textbf{Case 1. } $\beta$ contains an edge of $Y$ incident to $\bfd\lambda$.
Then, $$d(\bfc,C)\le d(\bfc,\bfd\lambda)\le D\cdot d_*(\bfc,\bfd\lambda)\le T_3  = \bdelta(D,R).$$

\textbf{Case 2. } $\beta$ does not meet $\bfd\lambda$. In this case, $\beta \subseteq \lambda$. 
Consider $\beta' = \beta \cap \calB_{T_2}(\bfc)$.
The subtrack $\beta'$ is minimal (since $\beta$ is) and so it is $T_1$-taut. 
Therefore $\bfc \in \conv_{T_1}(\bfd \beta')$. Hence there exist $\Delta_1,\Delta_2\in \bfd\beta'$ such that 
$d(\bfc,[\Delta_1,\Delta_2])\le T_1$.
Consider the two edges $e_1,e_2$ in $\lambda$ such that $d(\bfc,[e_1,e_2])\le T_1$.

By the definition of $\beta'$ the coboundary $\bfd\beta'$ is supported outside $\calB_{T_2}(\bfc)$, so $d(\bfc,\bfc'),d(\bfc,\bfc'')\ge T_2$.
We must have $$d(\bfc,C)\le T_2 = \bdelta(D)$$ as otherwise by \Cref{geodesic outside convex set} $d(\bfc,C) < \max\{d(\bfc',C),d(\bfc'',C)\}$ contradicting the maximality of $\bfc$.

In either case we get $d(\bfc,C)\le T:=\max \{T_2,T_3\} = \bdelta(D,R)$.
\end{proof}

\section{Thickness}
\label{sec: thick}

    In the previous section we saw that each locally minimizing track is taut. In particular, $|\lambda|$ can be bounded by  $|\conv(\bfd\lambda)|$.
    In this section, we will use the fact that $G$ is geometrically rigid to show that $|\conv(\bfd\lambda)|\le |\bfd\lambda|$.
    This is done by first showing that $\conv(\bfd\lambda)$ is ``barycentrically thick'' - which roughly says that it does not contain long bottlenecks. 

\subsection{Geometrically rigid hyperbolic groups}

\begin{definition}
    A hyperbolic group $G$ is \emph{geometrically rigid} if it is one-ended and $\partial G$ does not contain a cut pair.
\end{definition}

For a point $x_0\in X$ and $\xi\in \partial X$, $\rho\ge r\ge 0$ define the \emph{annulus}
$$A_r(x_0,\xi,\rho) = \{\zeta\in \partial X\;|\;\rho-r\le \gen{\zeta,\xi}_{x_0}\le \rho+r\}$$

\begin{lemma}\label{annuli in the boundary}
    Let $G$ be a geometrically rigid hyperbolic group. There exists $\Delta$ such that for all $x_0$ and $\xi\in \partial X$ and $\rho\ge r$ with $\rho -r\ge \bdelta$, the set $A_r(x_0,\xi,\rho)$ is path connected in $A_{r+\Delta}(x_0,\xi,\rho)$
\end{lemma}

\begin{proof}
    The set $A_r(x_0,\xi,\rho)$ can roughly be regarded as the difference of two balls in the visual metric.
    The lemma then follows from \cite[Lemma 6.1]{lazarovich2023commensurated}.
\end{proof}

\begin{lemma}\label{ball minus geodesic is connected}
    Let $G$ be a geometrically rigid hyperbolic group, then for all $\epsilon\ge0$, for all $\tau,R\ge \bdelta(\epsilon)$, for every bi-infinite geodesic $\gamma$ in $X$, and $x\in \gamma$, the subset $\calB_R(x)\ssm \calN_\tau(\gamma)$ is contained in a single connected component of $\calB_R(x) \ssm \calN_\epsilon(\gamma)$.
\end{lemma}

\begin{proof}
Let $\epsilon \ge 0$, and $R\ge \tau\ge 0$. Let $\gamma$ be a bi-infinite geodesic and  $x\in \gamma$. 

Let $z_1,z_2\in \calB_R(x)\ssm \calN_\tau(\gamma)$. 
Let $z_1',z_2'$ be nearest points to $z_1,z_2$ in $\gamma$ respectively.
Then, since $z_i'$ is $\bdelta$ close to the geodesic $[x,z_i]$ we have 
\begin{equation}\label{eq: distance x and projection of z} 
d(x,z'_i) \le  d(x,z_i)-d(z_i',z_i)+\bdelta\le R-\tau+\bdelta.
\end{equation}

By \cite[Lemma 3.1]{bestvina1991boundary} there exists a geodesic ray $\zeta_i$ starting at $z_i'$ such that $d(z_i,\zeta_i)\le \bdelta$.
Let $\xi\in \partial X$ be one of the two ends of the bi-infinite geodesic $\gamma$.
Consider a point $x_0\in \gamma$ at distance $\rho$ from $x$, away from $\xi$.
It follows that $$ |\gen{\zeta_i,\xi}_{x_0} - \rho|\le |d(x_0,z'_i) - \rho|+\bdelta\le d(x,z'_i)+\bdelta\le R-\tau+\bdelta$$
In other words, $\zeta_i \in A_{R-\tau}(x_0,\xi,\rho).$

By \Cref{annuli in the boundary}, for $\rho -(R-\tau) \ge \bdelta$ there exists $\Delta$ such that $\zeta_1,\zeta_2$ are connected by a path $\Gamma$ in $A_{(R-\tau)+\Delta}(x_0,\xi,\rho).$

Let $A = \pi_{\calB_R(x)}(\Gamma)$. 
By \Cref{projection of connected is quasi connected}, the projection $\calN_\bdelta(A)$ is connected.
In fact, using $\bdelta$-thinness of triangles it is not hard to see that $A'=\calN_{\bdelta}(A)\cap \calB_R(x)$ is connected.

Note that all points in $A$ are contained in the sphere $\calS_R(x) = \{y\;|\;d(x,y)=R\}$.
Therefore, the intersection $A'\cap \calN_\epsilon(\gamma)$ is contained in the $\epsilon+\bdelta$ neighborhood of the two points $\{\gamma_-,\gamma_+\}=\gamma \cap \calS_R(x).$
However, since $\pi_\gamma(\Gamma) \subseteq \calB_{R-\tau+\Delta+\bdelta}(x)$ it follows that $\pi_\gamma(A')$ is contained in $\calB_{R-\tau+\Delta+\bdelta}(x)$.
Therefore, for $\tau\ge \bdelta(\epsilon)$, $A'\cap \calN_\epsilon(\gamma)=\emptyset$.

This shows that $\calB_R(x)\ssm \calN_\epsilon (\gamma)$ has a path between $z''_1= \pi_{\calB_R(x)}(\zeta_1)$ and $z''_2=\pi_{\calB_R(x)}(\zeta_2)$.

To finish the proof it suffices to show that $z_i$ can be connected to $z''_i$ in $\calB_R(x) \ssm \calN_\epsilon(\gamma)$.
We do it in 3 steps:
\begin{enumerate}
    \item $z_i$ is $\bdelta$ close to some point $\zeta_i$.
    Therefore, it can be connected to a point $w_i\in \zeta_i$ by a path of length $\bdelta$ that stays in $\calB_R(x)$.
    This path avoids $\calN_{\tau-\bdelta}(\gamma)$, and so for $\tau\ge \epsilon+\bdelta$ it avoids $\calN_{\epsilon}(\gamma)$.
    \item Using a subpath of $\zeta_i$ we can connect $w_i$ to the point $w'_i \in \zeta_i \cap \calS_R(x)$.
    Clearly, this path stays in $\calB_R(x)$
    \item Finally, the points $w'_i$ and $z''_i$ are $\bdelta$-close, and so they can be connected by a path of length $\bdelta$ in $\calB_R(x).$ \qedhere
\end{enumerate} 
\end{proof}

\subsection{Local minimizers are plump}

\begin{definition}\label{def: plump}
    For $L\ge \epsilon\ge 0$, a set $C$ is \emph{$(L,\epsilon)$-slender} at $x\in C$ if $\calB(x,L)\cap C \subseteq \calN_\epsilon(\gamma)$ for some geodesic segment $\gamma$.
    Otherwise we say $C$ is \emph{$(L,\epsilon)$-plump} at $x$.
\end{definition}

\begin{lemma}\label{local minimizers are plump}
    For $\epsilon\ge 0$, $D\ge \bdelta$, $L\ge \delta(D,\epsilon)$, $R\ge \bdelta(D,L,\epsilon)$ if $\lambda$ is a subtrack of the $R$-locally minimal locus $\calF_R$, then for all $\bfc\in \lambda$ such that $\bfc\notin \calN_R(\bfd\lambda)$ the subtrack $\lambda$ is $(L,\epsilon)$-plump at $x$.
\end{lemma}

\begin{proof}
Let $\epsilon \ge 0$, and let $D\ge \bdelta$ be such that every ball in $X$ is contractible.
Denote $B = \calB(\bfc,L)$
Assume that $\lambda' := \lambda\cap B\subseteq \calN_\epsilon (\gamma)$ for some geodesic $\gamma$.
We will derive a contradiction for $L\ge \bdelta(D,\epsilon)$ and $R\ge \bdelta(D,L,\epsilon)$.

View $\lambda'$ as a 1-cochain in $C_c^1(B)$. 
If $R\ge \bdelta(D,L)$ then $\lambda'$ is minimal in $Y$ and has no coboundary in $B$.
Therefore, $\lambda'$ is a 1-cocycle in $B$.
Since $B$ is contractible, there exists $\nu\in C^0_c(B)$ such that $\lambda' =\bfd\nu$ in $B$.
Since $B$ is compact,  the complement $\nu' = B^0\ssm \nu\in C^0_c(B)$ also satisfies $\lambda' = \bfd \nu$ in $B$.

By \Cref{ball minus geodesic is connected}, there exists $\tau = \bdelta(\epsilon)$ such that $B\ssm \calN_\tau(\gamma)$ is connected in $B\ssm \calN_\epsilon(\gamma)$.
One of $\nu,\nu'$, say $\nu'$, meets $B\ssm \calN_\tau(\gamma)$.
Since $\bfd \nu'$ is in $\calN_\epsilon(\gamma)$, we deduce that $\nu'$ contains $B^0\ssm \calN_\tau(\gamma)$.
In other words, $\nu$ is contained in  $\calN_\tau(\gamma)$.

Now, consider $\nu\in C_c^0(Y)$.
Its coboundary $\bfd\nu$ consists of two parts, $\lambda'$ (its coboundary in $B$) and $\lambda'':=\bfd\nu - \lambda'$ (the parts of its coboundary in $Y$ which lie in $Y\ssm B$).
Let us estimate the size of $\lambda'$ and $\lambda''$.
The subtrack corresponding to $\lambda'$ and containing $\bfc$ contains a path connecting $\bfc$ to the boundary of the ball $B$, so $|\lambda'| \ge  L/D$.
To estimate the size of $\lambda''$, note that every edge of $\lambda''$ is incident to a vertex in $B\cap \calN_\tau(\gamma)$ and to a vertex in the complement of $B$. 
Therefore,  $|\lambda''|\le \bdelta(D,\tau) = \bdelta(D,\epsilon)$.

Then, $\bfd \lambda'' = \bfd \lambda'$, and if $L\ge \bdelta(D,\epsilon)$ we get $$|\lambda''|  \le \bdelta(D,\epsilon) \le L/D \le |\lambda'|,$$
contradicting the minimality of $\lambda'$.
\end{proof}

\subsection{Barycentrical thickness of coboundaries}

\begin{definition}\label{def: barycenterically thick}
Let $M\ge 0$. 
A point $x$ is an $M$-barycenter of $y_1,y_2,y_3$ if it is at distance $M$ from the three sides of a geodesic triangle spanned by $y_1,y_2,y_3$.

A set $A\subseteq X$ is \emph{$M$-barycentrically thick} if 
every point in $\conv(A)$ is an $M$-barycenter of a geodesic triangle spanned by points in $A$.
\end{definition}

\begin{remark}
    Roughly speaking, the following are equivalent:
    \begin{enumerate}
        \item $A$ is barycentrically thick
        \item The nearest point projection of $A$ to any geodesic is coarsely connected -- that is, it cannot be separated into two sets whose projections to the geodesic are far.
        \item The convex hull $\conv(A)$ is plump away from points of $A$.
    \end{enumerate}
\end{remark}  

We have already proved that local minimizers are taut (\Cref{local minimizers are taut}) and plump away from their coboundary  (\Cref{local minimizers are plump}).
In view of the equivalence above, we must have that the coboundary of local minimizers is barycentrically thick.
We make this precise in the next lemma.

\begin{lemma}\label{coboundary of local minimizers is thick}
    For $D\ge\bdelta$, $R\ge \bdelta(D)$, $M\ge \bdelta(D,R)$ if $\lambda$ is a subtrack of $\calF_R$ then $\bfd\lambda$ is $M$-barycentrically thick.
\end{lemma}

\begin{proof}
    Let $C = \conv(\bfd \lambda)$. 
    Consider $x\in C$. 
    Our goal is to show that $x$ is an $M$-barycenter of the points in $\bfd\lambda$. 
    
    \begin{claim*}
    For $D\ge \bdelta, \epsilon\ge \bdelta (D), L \ge \bdelta(D,\epsilon)$ and $R\ge \bdelta(D,\epsilon,L)$ either $d(x,\bfd\lambda)\le R$ or $C$ is $(L,\epsilon)$-plump at $x$. 
    \end{claim*}
    \begin{proof}[Proof of Claim:]

    Assume that $d(x,\bfd\lambda)\ge R$ and $C$ is $(L,\epsilon)$-slender at $x$. 
    For $\epsilon \ge \bdelta$ and $L\ge \bdelta( \epsilon)$, the ball $\calB(x,100\epsilon)\cap C$ disconnects $C$ into two components $C_-,C_+$. 
    Each of $C_-,C_+$ must contain points of $\bfd\lambda$. 
    Since $\lambda$ is connected, it therefore must contain some $\bfc\in \lambda \cap \calB(x,100\epsilon)$.

    By \Cref{local minimizers are taut}, there exists $T=\bdelta(D)$ such that $\lambda \subseteq \calN_T(C)$.

    If $C$ is $(L,\epsilon)$-slender then $\lambda$ is $(L-100\epsilon,T+\epsilon)$-slender at $\bfc$.
    For large enough $L$ and $R$ this is impossible by \Cref{local minimizers are plump}.
    \end{proof}

    We divide the rest of the proof of the lemma into the two possible cases:

    \noindent\textbf{Case 1.}   $d(x,\bfd\lambda)\le R$:

    Say $d(x,y)\le R$ for some $y\in \bfd \lambda$. Then for $M\ge R$,  $x$ is an $M$-barycenter of the (very) degenerate geodesic triangle spanned by $y,y,y$.

    \noindent\textbf{Case 2.}  $d(x,\bfd\lambda)\ge R$ and 
     and $C$ is $(L,\epsilon)$-plump at $x$:

    By definition of the convex hull $x\in [x_-,x_+]$ for some $x_-,x_+\in \bfd\lambda$.
    Since $C$ is $(L,\epsilon)$-plump, there exists some point $y\in C\cap \calB(x,L) \ssm \calN_\epsilon([x_-,x_+])$.
    Again, by definition of $C$, $y\in [y_-,y_+]$ for some $y_-,y_+\in \bfd\lambda$.

    Let $m_-\in [x_-,x_+]$ be the barycenter of the triangles $y_-,x_-,x_+$. That is, $m_-$ is the point on $[x_-,x_+]$ at distance $\bdelta$ from both $[x_-,y_-]$ and $[x_+,y_-]$.
    Similarly, define $m_+\in [x_-,x_+]$ to be a $\bdelta$-barycenter of the triangle $y_+,x_-,x_+$.
    
    By considering the points $x,y$ in the $\bdelta$-thin quadrilateral $x_-,x_+,y_-,y_+$, one sees that if $y$ is not in the $\bdelta$-neighborhood of $[x_-,x_+]$ then 
    $$\min\{d(x,m_-),d(x,m_+)\}\le d(x,y)+\bdelta\le L+\bdelta = \bdelta(L).$$ 

    So $x$ is a $\bdelta(L)$-barycenter of the triangle spanned by $x_-,x_+$ and one of $y_-,y_+$.  
\end{proof}



    

    


\subsection{The convex hulls of barycentrically thick sets}

\begin{proposition}\label{isoperimetric ineq for thick sets}
    For $M\ge 0$  if $A$ is $M$-barycentrically thick then $|\conv(A)|\le \bdelta(M) |A|.$
\end{proposition}

We begin by proving a variant of this proposition for trees: If $T$ is a finite tree, and $V_1(T)$ is its set of leaves, then $T = \conv(V_1(T))$. Let $V_{\ge 3}(T)$ be the set of all vertices $v$ of degree $\deg(v)\ge 3$.
Then, the set $V_1(T) \cup V_{\ge 3}(T)$ is exactly the set of $0$-barycenters of points in $V_1(T)$. 
The following lemma proves the proposition in this case.

  \begin{lemma}\label{isoperimetric inequality for a tree}
        If $T$ is a finite tree,  then $|V_{\ge 3}(T)| \le  |V_1(T)|$.
    \end{lemma}

    In particular, if $V_1(T)$ is $0$-barycentrically thick (in $T$), then $V(T) = V_1(T) \cup V_{\ge 3}(T)$, and so $|V(T)|\le 2|V_1(T)|$.

    \begin{proof}
        The Euler characteristic of the tree is $\chi(T)=1\ge 0$.
        Define \[\chi(v) = 1-\tfrac12 \deg(v) \le \begin{cases} \tfrac12 &\deg(v)=1 \\ 0 &\deg(v)=2 \\ -\tfrac12 &\deg(v)\ge 3\end{cases}\]
        Then,
        \[
        0\le 1 = \chi(T) = \sum_{v\in V(T)} \chi(v)  \le \tfrac12 |V_1(T)| - \tfrac12 |V_{\ge 3}(T)|
        \]
        Thus  $|V_{\ge 3}(T)| \le |V_1(T)|.$
    \end{proof}
    
The way we will prove \Cref{isoperimetric ineq for thick sets} is by creating an approximating tree:
    
    \begin{lemma}\label{approximating tree}
        Let $A$ be an $M$-barycentrically thick set. Then, there exists a finite tree $T$ and $a_0\in V(T)$ such that:
    \begin{enumerate}[label=(T\arabic*)]
        \item \label{T has no deg 2} All vertices $V(T)\setminus \{a_0\}$ do not have degree 2 -- that is, $V(T) \setminus \{a_0\} \subseteq  V_1(T)\cup V_{\ge 3}(T)$,
        \item \label{leaves of T}  $|V_1(T)|\le |A|$, and
        \item \label{T covers C}  $|\conv(A)| \le \bdelta(M)|V_{\ge 3}(T)|$.
    \end{enumerate}
    \end{lemma}

\begin{proof}
    Let $A$ be an $M$-barycentrically thick set, and let $C = \conv(A)$.
    Fix some $a_0\in A$.
    
    For $R> 0$, let $n:C\ssm\{a_0\} \to \bbZ_{\ge 0}$ be the map defined by $n(x) = \lceil\tfrac{d(x,a_0)}{R}\rceil$ for $x\in C\ssm \{a_0\}$. That is, $n(x)$ is the maximal $n\ge 0$ such that $nR < d(x,a_0)$.

    Consider the set $S_n = \calS(a_0,n\cdot R)\cap C$ for all $n\ge 0$, and their union $S = \bigcup_n S_n$. Let $p:C\ssm\{a_0\}\to S$ be the map that sends $x$ to the unique point of intersection of a geodesic $[a_0,x]$ with $S_{n(x)}$,  i.e.\ $p(x)$ is the nearest point projection of $x$ to a strictly smaller sphere in the collection $S_n$.
    Extend $p$ to a map $p:C\to S$ by setting $p(a_0)=a_0$.

    For $r> 0$, choose a maximal set $V_n$ in $S_n$ such for all $x, y\in V_n$ if $x\ne y$ then $\calB(x,r)\cap \calB(y,r)=\emptyset$.
    Since $V_n$ is maximal, every point in $S_n$ is at distance at most $2r$ from some point in $V_n$.
    Let $V = \bigcup_n V_n$ and let $\pi:C  \to V$ be the map that sends $x\in C\ssm \{a_0\}$ to a point in $V_{n(x)}$ at distance $2r$ away from $p(x)$.

    The following claim shows that iterative applications of $p$ and $\pi$ are uniformly close.

    \begin{claim}\label{claim: distance inequality p pi}
        For $r'\ge r\ge \bdelta$, and $R\ge \bdelta(r')$, if $x,y\in S_n$ and $d(x,y)\le r'$ then 
        \begin{align}
            \label{eq: d(p,p)}d(p(x),p(y)) &\le \bdelta,\\
            \label{eq: d(p,pi)}d(p(x),\pi(y))&\le 2r+\bdelta\\ 
            \label{eq: d(pi,pi)}d(\pi(x),\pi(y))&\le 4r+\bdelta, \text{ and}\\
            \label{eq: d(p j,pi j)} d(\pi^j(x),p^j(x))&\le r',\;  \forall j\in \bbN,
        \end{align}
        Moreover, if $a\in A$, $x\in V_n$,
        \begin{equation}\label{eq: close p implies equal pi}
        \text{ if }d(p^i(a),x')\le r/2 \text{ then } \pi^i(a)=x'
        \end{equation}
    \end{claim}
    \begin{proof} 
        Indeed, since triangles are thin, and $R\gg r'$ we have that if $d(x,y)\le r'$ then $d(p(x),p(y))\le \bdelta$. If $r' \ge \bdelta(r)$ we also have $$d(p(x),\pi(y))\le 2r+\bdelta \le r'$$
        and also 
        \begin{equation*}
            d(\pi(x),\pi(y))\le 4r+\bdelta\le r'. 
        \end{equation*}
        Iterating these gives \eqref{eq: d(p j,pi j)}.
        
        Finally, to prove \eqref{eq: close p implies equal pi}, let $a\in A$ and assume that $d(p^i(a),x')\le r/2$ for $x'\in V_n$. Then, by \eqref{eq: d(p j,pi j)} $d(p^{i-1}(a),\pi ^{i-1}(a))\le r'$. By \eqref{eq: d(p,p)}, $d(p^i(a),p(\pi^{i-1}(a)))\le \bdelta$. Therefore, for $r\ge \bdelta$, $d(x',p(\pi^{i-1}(a)))\le r/2+\bdelta\le r$. 
        It follows by the definition of $\pi$ that $\pi^i(a)=x'$.
    \end{proof}

    Let $V' = \bigcup_{n\ge 2} \pi^n(A)$, and set $A' := A\ssm\{a_0\}$.
    Consider the tree $T$ whose set of vertices is the abstract disjoint union $A'\sqcup V'$, and edges connect $x \in A'$ to $\pi^2(a)\in V'$ and the vertex $x\in V'$ to $\pi(x)\in V'$.
    
    Clearly, $T$ is a tree.

    We claim that for $r\ge \bdelta(M)$ and $R\ge \bdelta(r)$ the tree $T$ satisfies \ref{T has no deg 2}-\ref{T covers C}.

\bigskip

    \noindent\textbf{\ref{T has no deg 2} and \ref{leaves of T}:}
    Clearly every vertex in $A'$ is a leaf. 
    It remains to show that every vertex of $V' \ssm \{a_0\}$ has degree 3.
    Let $x\in V'\setminus \{a_0\}$, then since $A$ is $M$-barycentrically thick, $x$ is an $M$-barycenter of some $a_1,a_2,a_3\in A$. 
    By considering the $\bdelta$-thin quadrilateral $a_0,a_1,a_2,a_3$, we see that $x$ is an $(M+\bdelta)$-barycenter of $a_0$ and two of the three points $a_1,a_2,a_3$. Without loss of generality, $x$ is an $(M+\bdelta)$-barycenter of $a_0,a_1,a_2$.
    It follows that $d(x,p^{i_1}(a_1))\le M+\bdelta$ and $d(x,p^{i_2}(a_2))\le M+\bdelta$ for some $i_1,i_2\in \bbN$.
    By \eqref{eq: close p implies equal pi}, for appropriate $R\gg r\gg M+\bdelta$ we have $x  = \pi^{i_1} (a_1)=\pi^{i_2}(a_2)$.
    By definition of $V'$ we may assume that at least one of $\max\{i_1,i_2\}\ge 2$.
    For $j=1,2$, let $$y_j = \begin{cases} \pi^{i_j-1}(a_j) & i_j\ge 3\\ a_j &i_j =1, 2. \end{cases}$$ 
    Note that $d(y_1,y_2)>R$, so that $y_1\ne y_2$. 
    This shows that $x$ has degree three, since $y_1,y_2,\pi(x)$ are three distinct neighbors of $x$.

    \bigskip

    \noindent \textbf{\ref{T covers C}:} 
    Let $x\in C$, then $x\in [x_-,x_+]$ for some $x_-,x_+\in A$. Since triangles are thin, $x$ is at distance $\delta$ from a geodesic $[a,a_0]$ where $a\in \{x_-,x_+\}\subseteq A$.
Therefore, $x$ is at distance at most $2R+\delta$ away from some point $p^i(a)$ for $i\ge 2$. By \eqref{eq: d(p,pi)} $d(p^i(a),\pi^i(a)) \le r'$. 
Therefore, $C \subseteq \calN_{R+\delta+r'}(V')$.
    It follows that $|C|\le \bdelta(R)|V'| $.
\end{proof}

\begin{proof}[Proof of \Cref{isoperimetric ineq for thick sets}]
    Let $A$ be an $M$-barycentrically thick set, then by \Cref{approximating tree,isoperimetric inequality for a tree}, we have the desired inequality
    \begin{equation*}|\conv(A)|\le \bdelta(M) |V'| \le \bdelta(M)|V_{\ge 3}(T)|\le \bdelta(M)|V_1(T)| \le  \bdelta(M) |A|.\qedhere\end{equation*}
\end{proof}

\subsection*{Summary of \Cref{sec: thick}}

Combining the results of this section we obtain:

\begin{proposition}\label{the inequality for locally minimizing}
    For $D\ge \bdelta$ and $R\ge \bdelta(D)$ any subtrack $\lambda$ of $\calF_R$ satisfies 
    \begin{equation}\label{eq: inequality for locally minimizing}|\lambda| \le\bdelta(D,R)|\bfd \lambda|.
    \end{equation}
\end{proposition}

\begin{proof}
    By \Cref{local minimizers are taut}, for $D\ge \bdelta$ and $R\ge \bdelta(D)$, the track $\lambda$ is $T=\bdelta(D)$-taut. By \Cref{coboundary of local minimizers is thick}, for $D\ge \bdelta$ and $R\ge \bdelta(D)$, the coboundary $\bfd\lambda$ is $M=\bdelta(D,R)$-barycentrically thick. 
    By \Cref{isoperimetric ineq for thick sets}, we get 
    \begin{equation*}|\lambda| \le |\calN_T(\conv(\bfd\lambda))|\le \bdelta(T)|\conv(\bfd\lambda)|\le \bdelta(T)\bdelta(M)|\bfd\lambda| = \bdelta(D,R)|\bfd\lambda|.\qedhere\end{equation*}
\end{proof}

\subsection*{Positivity of the Cheeger constant} \label{sec: theorem b}
We will now restate and prove \cref{thm: positive 1-cheeger constant}.

\cheegertheorem*

\begin{proof}[Proof of \Cref{thm: positive 1-cheeger constant}]
Let $G$ act freely and cocompactly on a simply connected simplicial complex $Z$. 
Let $Y$ be the 2-skeleton of the Rips complex $\Rips_D(Z)$ for some $D$ to be determined later.
We start by first showing that it suffices to prove the theorem for $Y$:

\begin{claim}
    For $D\ge 1$,  $h^1(Y)>0 \implies  h^1(Z)>0$.
\end{claim}

\begin{proof}[Proof of Claim.]
There exists a $G$-equivariant embedding $i:Z \inj Y$ and since $Z$ is simply connected there is a $G$-equivariant retraction $r:Y\to Z$ such that $r\circ i=\id_{Z}$.
The preimage of each cell of $Z$ under $r$ has at most $\bdelta(D)$ cells.

Let $\alpha_1\in B_c^2(Z)$, then $\alpha_2= r^\sharp(\alpha_1)\in B_c^2(Y)$ and $|\alpha_2|\le \bdelta(D) |\alpha_1|$. 
If $h^1(Y)>0$ then there exists $\lambda_2\in C^1_c(Y)$ such that $\alpha_2=\bfd \lambda_2$ and $|\lambda_2|\le \bdelta(D) |\alpha_2|\le \bdelta(D) |\alpha_1|$.
Set $\lambda_1 = i^\sharp(\lambda_2)$, then $$\bfd \lambda _1 = \bfd(i^\sharp \lambda_1)=i^\sharp\bfd(\lambda_1)=i^\sharp \alpha_2 = i^\sharp r^\sharp \alpha_1 = (r\circ i)^\sharp\alpha_2=\alpha_2$$ and
$$|\lambda_1|\le \bdelta(D)|\lambda_2|\le \bdelta(D) |\alpha_1|.\qedhere$$
\end{proof}


To prove the theorem we need to show that if $\alpha\in C^1_c(Y)$ is minimal then $|\alpha|\le \bdelta |\bfd \alpha|$. 
We can view $\alpha$ as a singular pattern $\calF$ on $Y$ by placing a connector on each edge in $\alpha$, and connecting them with a regular segment in a 2-simplex if exactly two out of the three edges of the simplex are in $\alpha$, and with singular segments otherwise\footnote{This is essentially the reverse procedure to the one described in \Cref{section: tracks cochains}.}.

Since $\alpha$ is minimal it is $R$-locally minimal (for all $R$), i.e. $\calF = \calF_R$. By \Cref{the inequality for locally minimizing}, for $D\ge \bdelta$ and $R\ge \bdelta(D)$ we have the desired inequality $|\lambda| \le \bdelta(D,R)|\bfd \lambda|$ for each track $\lambda\subseteq \calF$. Summing over all the tracks we get:
\begin{equation*}|\alpha| = \sum_{\lambda\subseteq \calF} |\lambda|\le \bdelta(D,R) \sum _{\lambda\subseteq \calF}|\bfd \lambda| =\bdelta(D,R)|\bfd \alpha|\qedhere\end{equation*}
\end{proof}

\section{Proof of \Cref{main theorem}}
\label{sec: pf of main theorem}

\begin{lemma}\label{all graphs are equivalent}
Let $G$ act freely and cocompactly on two connected graphs $X,X'$. There exists a constant $C=C(X,X')$ such that if $\phi:H\to G$ is an injective homomorphism, and $\Phi:X\to X$ is $\phi$-equivariant then there exists a $\phi$-equivariant map $\Phi':X'\to X'$ such that 
$$\tfrac 1C\TD(\Phi)\le \TD(\Phi')\le C\cdot \TD(\Phi)$$
and 
$$\tfrac 1C\AD(\Phi)\le \AD(\Phi')\le C\cdot \AD(\Phi)$$
\end{lemma}
\begin{proof}
        Fix some $G$-equivariant maps $\Psi :X\to X'$ and $ \Psi':X'\to X$. Let $$K_+=\max_{e}\ell(\Psi(e)) \quad \text{and}\quad  K_-=\max _{e'} |\{e\;|\;e'\subset\Psi(e)\}|$$ where $e,e'$ run over the edges of $X,X'$ respectively. 
        Similarly, define $K_+',K_-'$ for the map $\Psi'$.
        
        For a $\phi$-equivariant map $\Phi:X\to X$ one can define the $\phi$-equivariant map $\Phi'=\Psi\circ \Phi \circ \Psi':X'\to X'$.
        For an edge $e'$ of $X'$  we have 
        $$\ell(\Phi'(e'))\le \sum _{e\subset \Psi'(e')}\ell(\Psi\Phi(e))\le K_+ \sum _{e\subset \Psi'(e')}\ell(\Phi(e))$$
        Therefore, 
        \begin{multline*}
            \TD(\Phi')= \sum_{e'}\ell(\Phi'(e'))=K_+\sum_{e'}\sum_{e\subset \Psi'(e')}\ell(\Phi(e))\\
            =K_+\sum_{e}\sum_{e\subset \Psi'(e')}\ell(\Phi(e))\le K_+K'_-\sum_e\ell(\Phi(e)) = K_+K'_- \TD(\Phi)
        \end{multline*}
        where the sums are over edges $e,e'$ of $X/H$ and $X'/H$ respectively.
        Similarly, $$\TD(\Phi)\le K_+'K_-\TD(\Phi').$$

        Finally, let $k_0,k_0'$ denote the number of edge orbits in $X/G,X'/G$ respectively. Then \begin{multline*}
        \AD(\Phi') = \tfrac{1}{k_0'[G:H]}\TD(\Phi')
        \le k_0K_+K'_-\tfrac1{k_0[G:H]}\TD(\Phi)\\
        \le k_0K_+K'_-\AD(\Phi)\end{multline*}
        and 
        $$\AD(\Phi) \le k_0'K_+'K_-\AD(\Phi')\qedhere$$ 
\end{proof}
We are now ready to prove the main theorem of this paper, that we first restate:
\maintheorem*
\begin{proof}
Let $G$ be a geometrically rigid, residually finite, hyperbolic group. Let $X$ be a graph with a free and cocompact action of $G$. \Cref{all graphs are equivalent} shows that up to multiplicative error the total and average distortions do not depend on the graph $X$, and so we may assume that $X$ is the graph given in \Cref{the bicombing} with the bicombing $q$.

For $D\ge \bdelta$, the 2-skeleton $Y=Y_D$ of the $D$-Rips complex $\Rips_D(X)$ is simply connected.

Since $G$ is residually finite, for every $D\ge 0$ and $\rho\ge 0$ we can find a finite index subgroup $\dot G = \dot G_{D,\rho}$ such that $\injrad(Y/\dot G)\ge \rho$.
It follows that for every finite index subgroup $H\le \dot G$ we have $\injrad(\overline Y)\ge \rho$ where $\overline Y = Y/H$.

Consider an injective homomorphism $\phi:H\to G$, let $\Phi:Y\to X$ be the $\phi$-equivariant map that minimizes $\TD_Y(\Phi)$, and in particular sends edges to geodesics paths. Let $\overline\calF$ be the associated foliation on $\overline Y$, and $\calF_R$ its $R$-locally minimizing locus.

By \Cref{the inequality for locally minimizing}, for $D\ge \bdelta$ and $R\ge \bdelta(D)$, we have
\begin{equation}\label{eq: 2}
|\calF_R|=\sum_{\overline\lambda \subseteq \calF_R}|\overline\lambda|\le \bdelta(D,R)\sum _{\overline\lambda \subseteq \calF_R}|\bfd\overline\lambda|.
\end{equation}
where the sum is over the components of $\calF_R$.

By extending slightly each component $\overline\lambda \subseteq \overline\calF_R$ we can make it into a track with dead-ends whose dead-ends are in the interior of segments of $\overline\calF \setminus \overline\calF_R.$ 
Therefore, 
\begin{equation} \label{eq: 3}
    \sum _{\overline\lambda\subseteq \calF_R}|\bdelta \overline\lambda|\le \bdelta(D) |\overline\calF \setminus \overline\calF_R|.
\end{equation}

By \Cref{complement of the minimizing locus}, for $D\ge \bdelta$, $R\ge \bdelta(D)$ and $\rho\ge \bdelta(D,R)$, we have
\begin{equation}\label{eq: 1}
    |\overline\calF \setminus \overline\calF_R|\le \bdelta(D,R)\cdot \vol(\overline Y)
\end{equation}
since $\injrad (\overline Y)\ge \rho$.




Combining the inequalities \eqref{eq: AD < F}, \eqref{eq: 1}, \eqref{eq: 2}, \eqref{eq: 3} and \eqref{eq: vol < index}, we get 
\begin{multline*}
\TD(\Phi)\le |\overline\calF|= |\overline\calF \setminus \overline\calF_R|+|\overline\calF_R|\\\le \bdelta(D,R)|\overline\calF \setminus \overline\calF_R|\le \bdelta(D,R)\vol(\overline Y) \le \bdelta(D,R)[G:H]
\end{multline*}
So, we get the desired inequality
$$\AD(\Phi) \le \bdelta(D,R).$$

\bigskip

Now, if $\phi:H\to H'$ is an isomorphism between two finite index subgroups, then by the proof of Theorem B in
\cite[Inequality (8.3)]{lazarovich2025finite} 
if $D\ge \bdelta$ is such that $\Rips_D(X)$ is contractible, the number of tracks in $\calF$ is bounded below by the index
$$[G:H'] \le \bdelta(D) \#\{\text{tracks in }\overline\calF\}\le \bdelta(D)|\overline\calF|.$$

By property \ref{quasi geodesic} and the construction of the singular pattern $\calF$ it follows that 
$$|\overline\calF| \le \bdelta \TD_Y(\Phi)\le \bdelta(D) \TD(\Phi).$$
By finite index rigidity\footnote{A group is \emph{finite index rigid} if it does not contain isomorphic finite index subgroups of different indices.} of hyperbolic groups \cite[Theorem A]{lazarovich2025finite} the isomorphic subgroups $H,H'$ have the same index, and we get 
$$[G:H]=[G:H'] \le \bdelta (D) \TD(\Phi).$$
Or equivalently, 
\begin{equation*}\frac 1{\bdelta(D)}\le \AD(\Phi).\qedhere\end{equation*}
\end{proof}

\begin{remark}
Without the finite index rigidity of hyperbolic groups \cite{lazarovich2025finite} the proof shows that $\tfrac 1 \bdelta [G:H']\le \TD(\Phi)\le \bdelta [G:H]$ and in particular that if $H,H'\le \dot G$ are isomorphic finite index subgroups then
    $ [G:H']\le \bdelta [G:H].$
    The next lemma shows that finite index rigidity follows.
\end{remark}

\begin{lemma}
    Let $G$ be a group. Assume that there exists $C\ge 1$ and a finite index subgroup $\dot G\le G$ such that for all isomorphic finite index subgroups $H,H'\le \dot G$ 
    \begin{equation}\label{index bounds} [G:H']\le C\cdot  [G:H].
    \end{equation}
    Then, $G$ is finite index rigid.
\end{lemma}
\begin{proof}
    Define the \emph{max index of $H$ in $G$} to be $$\MI(H,G) := \sup\{[G:K]\;|\;H\simeq K\le G, [G:K]<\infty\}$$
    and the \emph{max index gradient} to be
    $$\MIG(H,G) := \limsup _{K\le H, [H:K]<\infty} \frac {\MI(K,G)}{[H:K]}$$
    where the limsup is taken over the directed set of finite index subgroups of $H$.
    If $H,H'\le G$ are isomorphic finite index subgroup then it is easy to verify that 
    \begin{equation}\label{equality of MIG}[G:H']\MIG(G,G)=\MIG(H',G)=\MIG(H,G) = [G:H]\MIG(G,G)
    \end{equation}
    The inequality \eqref{index bounds} shows that $0<1 \le \MIG(G,G)\le C<\infty .$ 
    Dividing both sides of \eqref{equality of MIG} shows that  $[G:H]=[G:H']$.
\end{proof}

\section{Applications}
\label{sec: applications}

\subsection{Commensurators of manifolds}
    The goal of this subsection is to prove \Cref{main theorem manifolds}. 
    We begin with a technical lemma about mapping a standard Euclidean simplex to a CAT(0) space:

    Let $e_0,\dots,e_n$ be the standard basis vectors of $\bbR^{n+1}$. Let $\Delta^n$ be the standard $n$-simplex $\Delta^n=[e_0,\dots,e_n]$.

    Let $X$ be a CAT(0) space, and let $x_0,\dots,x_n\in X$.
    Define $f=f_{x_0,\dots,x_n}:\Delta^n\to X$ by induction on $n$. For $n=0$, $\Delta^0=\{e_0\}$, and define $f(e_0)=x_0$.
    To define $f$ on $\Delta^n$, assume that $f|_{\Delta^{n-1}}$ was defined, and let $x\in \Delta^n$. Write $x=(1-t)x'+t e_n$ for $x'\in \Delta^{n-1}$ and $0\le t\le 1$. Let $\gamma_x:[0,1]\to X$ be the unique constant speed geodesic connecting $f(x')$ and $x_n$. Define $f(x)=\gamma_x(t)$.
    Note that by definition $f(e_i)=x_i$.
    
\begin{lemma}\label{lemma: Lipschitz on simplex}
    For all $n$ there exists $C_n\ge 1$ such that the map $f$ is $K$-Lipschitz for $K=C_n\cdot \sum_{0\le i<j\le n} d(x_i,x_j)$. 
\end{lemma}

\begin{proof}
    Let us prove this by induction on $n$. For $n=0$ there is nothing to prove.
    Let $n\ge 0$, assume that there exists $C_{n-1}\ge 1$ such that $f|_{\Delta^{n-1}}$ is $\left(C_{n-1}\cdot \sum_{0\le i<j\le n-1} d(x_i,x_j)\right)$-Lipschitz.

    Let $\theta_n$ be the smallest angle that a line through $\Delta^{n-1}$ and $e_n$ makes with $\Delta^{n-1}$, and let $C_n = \frac{C_{n-1}}{\sin(\tfrac12 \theta_n)}$.
    We will show that $f$ is $\left(C_n\cdot \sum_{0\le i<j\le n} d(x_i,x_j)\right)$-Lipschitz on $\Delta^n$.
    
    Let $u,v\in \Delta^n$.
    Write $u=(1-t)u'+te_n$ and $v=(1-s)v'+se_n$ with $0\le s,t\le 1$ and $u',v'\in \Delta^{n-1}$ and set $\theta :=\angle(e_n-v',v'-u')\ge \theta_n$.
    See \Cref{fig:mapping simplex}.

    We have
    \begin{figure}
        \centering
        \includegraphics[width=0.9\textwidth]{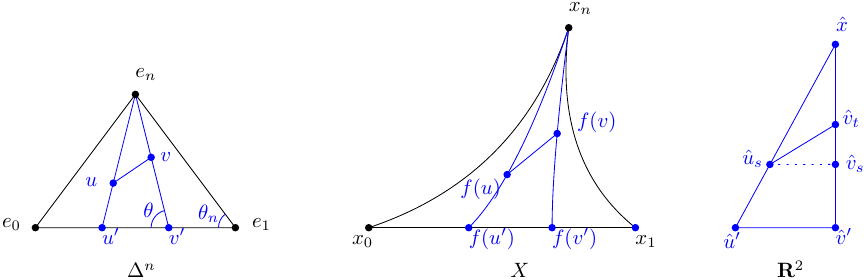}
        \caption{The simplex $\Delta^n$, its image in $X$ and the comparison triangle in Euclidean space.}
        \label{fig:mapping simplex}
    \end{figure}
    \begin{align*} 
    \|u-v\|&= \|(1-t)u'- (1-s)v'+(t-s)e_n\|\\
    &= \|(1-t)(u'-v')+(t-s)(e_n-v')\|\\
    &\ge \sin(\tfrac 12\theta)\cdot \left((1-t)\|u'-v'\|+|t-s|\|e_n-v'\|\right)
    \end{align*}
    where the last inequality we used the inequality $\|x-y\|\ge \sin(\tfrac12\angle(x,y))(\|x\|+\|y\|)$. 
    Using $\theta\ge \theta_n$ and $\|e_n-v'\|\ge d_{\bbR^n}(e_n,\Delta^{n-1})\ge 1$, we get
    \begin{equation}\label{lower bound euclidean}
        \|u-v\|\ge \sin(\tfrac12\theta_n)\cdot ((1-t)\|u'-v'\|+|t-s|)
    \end{equation}
    
    Consider the triangle $u',v',e_n$ in $\Delta^n$ and the geodesic triangle in $X$ spanned by their images $f(u'),f(v'),f(e_n)=x_n$. Note that two of the sides of this triangle are the geodesics $\gamma_u,\gamma_v$, and that by definition $f(u)=\gamma_u(t)$ and $f(v)=\gamma_v(s)$.

    Consider the comparison triangle $\hat u',\hat v',\hat x$ in $\bbR^2$ for $u',v',x_n$, and the points $$\hat u_t = (1-t)\hat u'+t\hat x,\quad  \hat v_t = (1-t)\hat v'+t\hat x, \quad \hat v_s=(1-s)\hat v'+s\hat v'$$ on this triangle. Then, 

    \begin{align*}
        d(f(u),f(v))&= d(\gamma_u(t),\gamma_v(s))\\
        &\le d(\hat u_t,\hat v_s) &\text{CAT(0)-ineq.}\\
        &\le d(\hat u_t,\hat v_t)+d(\hat v_t,\hat v_s) &\Delta\text{-ineq.}\\
        &\le (1-t)\|\hat u'-\hat v'\|+|s-t|\cdot \|\hat x-\hat v'\| &\gamma_v\text{ is unit speed}\\
        &= (1-t)d(f(u'),f(v'))+|s-t|\cdot d(x_n,f(v'))\\
        &\le (1-t)\|u'-v'\|C_{n-1}\sum_{0\le i<j\le n-1} d(x_n,x_i) &\text{IH}
        \\
        &\quad\quad + |s-t|\sum_{i=0}^nd(x_n,x_i) &\text{convexity of }d\\
        &\le C_{n-1}((1-t)\|u'-v'\|+|s-t|)\sum_{0\le i<j\le n}d(x_i,x_j)\\
        &\le \frac{C_{n-1}}{\sin(\tfrac12 \theta_n)} \left(\sum_{0\le i<j\le n}d(x_i,x_j)\right) \|u-v\| &\text{by \eqref{lower bound euclidean}}\\
        &=C_n\left(\sum_{0\le i<j\le n}d(x_i,x_j)\right)\|u-v\|
    \end{align*}
    Proving that $f$ is $\left(C_n\cdot \sum_{0\le i<j\le n}d(x_i,x_j)\right)$-Lipschitz.
\end{proof}

\begin{proof}[Proof of \cref{main theorem manifolds}]
Let $M$ be a closed Riemannian manifold of negative sectional curvature of dimension $n\ge 3$. 
The fundamental group $G = \pi_1(M)$ is hyperbolic. Since $n\ge 3$ the boundary $\partial G\homeo \bbS^{n-1}$ has no local cut points, and so $G$ is geometrically rigid.
Assume that $G$ is residually finite.

Let $T$ be a simplicial triangulation of $M$.
Let $\tild T$ be the lift of $T$ to the universal cover $\tild M$. 
Choose some ordering $v_1<\dots <v_m$ of the vertices of $T$. 
This ordering induces an ordering on the vertices of each simplex of $T$, and thus also on each simplex of $\tild T$.

Let $X$ be the 1-skeleton of $\tild T$.
The fundamental group $G$ acts freely on $X$ 
(and on $\tild M,\tild T$) by deck transformations.
Let $\dot G\le G$ be the finite index subgroup of \Cref{main theorem} corresponding to the action $G\actson X$, and let $\dot M$ be the finite cover of $M$ corresponding to $\dot G\le G$.

Given a homotopy equivalence $F:M_1\to M_2$ between a finite index cover $M_1$ of $\dot M$ and a cover $M_2$ of $M$, let $\phi=F_*:H_1=\pi_1(M_1)\to H_2=\pi_1(M_2)$ be the induced map on fundamental groups.
By \Cref{main theorem} there exists a $\phi$-equivariant map $\Phi:X\to X$ such that 
\begin{equation}\label{avgdist inequality for smooth proof}
    \AD(\Phi)\le C.
\end{equation}

For each simplex $\sigma=[v_0,\dots,v_k]$ of $\tild T$ (whose vertices are ordered $v_0<\dots<v_k$), let $i_\sigma:\Delta^n=[e_0,\dots,e_k]\to \sigma=[v_0,\dots,v_k]$ be the characteristic map.
We may assume that the map $i_\sigma$ is $K$-bi-Lipschitz where $K=K(M)$.
Let us define the map $\tild f:\tild M\to \tild M$ on each $n$-simplex $[v_0,\dots,v_n]$ to be the composition $f_\sigma\circ i_\sigma\ii $ where $f_\sigma := f_{\Phi(v_0),\dots,\Phi(v_n)}:\Delta^n\to \tild M$ is the map defined at the beginning of this subsection sending $e_i$ to $\Phi(v_i)$.
Note that $\tild M$ is indeed a CAT(0) space by the Cartan-Hadamard Theorem. 
Note also that $\tild f$ is $\phi$-equivariant, and so induces a map $f:M_1\to M_2$ which is homotopic to $F$.

By \Cref{lemma: Lipschitz on simplex}, the map $\tild f$ is 
$\left(\bdelta \sum_{0\le i<j\le n}\ell(\Phi([v_i,v_j])\right)$-Lipschitz on each $n$-simplex.
By Rademacher's Theorem $f$ is almost-everywhere differentiable, and on each simplex $[v_0,\dots,v_n]$
satisfies 
$$\int_{T^1[v_0,\dots,v_n] } \|D_x\Psi(v)\|\bfd\vol(x,v)\le \bdelta \sum_{0\le i<j\le n}\ell(\Phi([v_i,v_j]).$$
Summing over the simplices of $M_1$, 
and using that each edge of the $T$ is in $\bdelta$ $n$-simplices of $T$, we get $$\int_{T^1M_1 } \|D_xf(v)\|\bfd\vol(x,v)\le \bdelta \TD(\Phi)$$
Clearly, $\Vol(T^1M_1)\propto [G:H_1]$, and so by \eqref{avgdist inequality for smooth proof},
$$\tfrac 1{\Vol(T^1M_1)}\int_{T^1M_1} \|D_xf(v)\|\bfd\vol(x,v)\le \bdelta \AD(\Phi)\le \bdelta.$$
By the Meyers-Serrin Theorem we may approximate $f$ by a smooth function (in the $W^{1,1}$-Sobolev norm) which is homotopic to $f$ and will satisfy a similar average distortion bound.
\end{proof}

\subsection{Outer automorphisms}
\begin{lemma}\label{out characteristic}
    Let $K\le G$ be a characteristic subgroup of finite index with finite center. Then, the restriction map  $\Aut(G)\to \Aut(K)$ has finite kernel. 
    In particular, if $|\Out(K)|<\infty$ then $|\Out(G)|<\infty$.
\end{lemma}

\begin{proof}
Let $\calK = \ker(\Aut(G)\to\Aut(K))$. 
Every automorphism of $G$ induces an automorphism of $\Aut(G/K)$. If $\calL$ denotes the kernel of the map $\Aut(G)\to \Aut(G/K)$, then $\calL\le \Aut(G)$ has finite index, and so $[\calK:\calK\cap \calL]<\infty$. Therefore to show that $|\calK|<\infty$ it suffices  suffices to show that $|\calL\cap \calK|<\infty$. 

Let $\phi\in \calK\cap \calL$ then $\phi|_K=\id$, and for all $g\in G$ $\phi(g)=gk$ for some $k\in K$.
Let $t_1,\dots,t_n$ be coset representatives of $K\le G$, and let $k_1,\dots,k_n\in K$ be such that $\phi(t_i)=t_ik_i$. Then, since $K$ is normal in $G$ and $\phi|_K=\id$ we have for all $k\in K$ 
$$t_ikt_i\ii=\phi(t_ikt_i\ii) = t_ik_ikk_i\ii t_i\ii.$$
This implies that $k_i\in Z(K)$. 
Since $|Z(K)|<\infty$ there are finitely many options for choosing $k_1,\dots,k_n$. 
Since $\phi|_K$ and $\phi(t_1),\dots,\phi(t_n)$ determine $\phi$, we get that there are only finitely many such automorphisms $\phi$.
This finishes the proof that $|\calK|<\infty$.

The map $\Aut(G)\to \Aut(K)$ descends to a map $\Aut(G)/K\to \Out(K)$ which also has finite kernel. It follows that if $|\Out(K)|<\infty$ then $|\Aut(G)/K|<\infty$.
Finally, $$|\Out(G)|=|\Aut(G)/G|\le|\Aut(G)/K|<\infty.\qedhere$$ 
\end{proof}

\begin{proof}[Proof of \cref{paulin thm}]
Let $\dot G\le G$ be as in \Cref{main theorem}, and let $K\le \dot G$ be a finite index characteristic subgroup of $G$. Then by \Cref{main theorem}, there exists $C$ such that for all $\phi\in \Aut(K)$, there exists $\Phi:X\to X$ such that $\TD(\Phi) \le C[G:K]$. There are finitely many such maps $\Phi$ up to translation, and so finitely many automorphisms $\phi$ up conjugacy.
This shows that $|\Out(K)|<\infty$. 
Moreover, $K$ has finite center since it is a non-elementary hyperbolic group. 
By \Cref{out characteristic}, we have $|\Out(G)|<\infty$ as well.
\end{proof}

\subsection{Average distortion of translation lengths}
Recall that for $\gamma\in G$, $\ell(\gamma)=\lim_{k\to \infty}\tfrac 1k \|\gamma^k\|$. We remark that if $X$ is the Cayley graph of $G$, then for all $o\in X$
\begin{equation}\label{eq: general translation length equality}
    \ell(\gamma)=\lim_{k\to \infty} \tfrac 1k d(o,\gamma^ko).
\end{equation}

\begin{proof}[Proof of \cref{cor: AL}]
Let $G$ be a residually finite geometrically rigid hyperbolic group. 
Let $g\in G$ and $\phi:H\to G$ be an injective isomorphism from a finite index subgroup $H\le G$. 
We want to show that \begin{equation}\label{eq: desired AL inequality} \AL(\phi,g)\le \bdelta \ell(g).\end{equation}
Since $\AL(g,\phi)$ is invariant under restricting $\phi$ to a smaller subgroup, we may assume that $H\le \dot G$ where $\dot G\le G$ as in \cref{main theorem} and that $H$ is normal.
Replacing $g$ by $g^k$ multiplies both sides of \eqref{eq: desired AL inequality} by $k$, and so it suffices to prove the inequality for a power of $g$. 
By replacing $g$ by a power $g^k$, we may assume that $g \in H$ (and therefore also $tgt\ii\in H$ for all $t\in G$).

By \cref{main theorem}, there exists a $\phi$-equivariant map $\Phi:X\to X$ such that $\AD(\Phi)\le \bdelta$.
Let $t_1,\dots,t_n$ be coset representatives of $H$ in $G$, and let $e$ be an edge of $X$. 
Then,
$$\tfrac1n\sum_{i=1}^n \ell(\Phi(t_i e))\le \tfrac 1n\TD(\Phi)\le k_0\AD(\Phi)\le \bdelta$$
where $k_0$ is the number of edge orbits of $G$ on $X$.

Similarly, for every $x,y\in X$, let $f_1,\dots,f_r$ be the edges along the geodesic path between $x,y$.
Then,
\begin{multline}\label{eq: average of distances}
    \tfrac 1 n\sum_{i=1}^nd(\Phi(t_ix),\Phi(t_iy))\le 
\tfrac 1 n\sum_{i=1}^n \sum _{j=1}^r\ell (\Phi(t_if_j))\\
= \sum _{j=1}^r\left(\tfrac 1 n\sum_{i=1}^n \ell (\Phi(t_if_j))\right)\le \bdelta \cdot r=\bdelta\cdot d(x,y).
\end{multline}

Therefore,
\begin{align*}
\AL(\phi,g)&=\tfrac 1n \sum_{i=1}^n\ell(\phi(t_i g t_i\ii))\\
&= \tfrac 1n \sum_{i=1}^n \left(\lim _{k\to \infty}\tfrac 1kd(\Phi(t_ix),\phi(t_ig^kt_i\ii)\Phi(t_i x))\right) &\text{by \eqref{eq: general translation length equality}}\\
&=\lim_{k\to \infty}\tfrac 1k \left(\tfrac 1n \sum_{i=1}^n d(\Phi(t_ix),\Phi(t_ig^kx))\right)&\text{by }\phi\text{-equivariance}\\
&\le \lim_{k\to \infty}\tfrac 1k \bdelta d(x, g^k x) &\text{by \eqref{eq: average of distances}}\\
&=\bdelta \ell(g) &\text{by \eqref{eq: general translation length equality}}.\quad \qedhere
\end{align*}
\end{proof}

\subsection{Relative commensurators}

\begin{lemma}\label{exp divergence from geodesic}
    Let $X$ be a hyperbolic graph. Let $\alpha$ be an infinite geodesic. Let $\gamma$ be a path in $X$ with endpoints $\gamma_\pm$, let $x_\pm$ be nearest point projections of $\gamma_\pm$ to $\alpha$. If  $d(x_-,x_+)\ge \bdelta$ then $$\length(\gamma)\ge \exp(\tfrac 1 \bdelta d(\gamma,\alpha)) d(x_-,x_+)$$
    where $d(\gamma,\alpha)$ is the shortest distance between points of $\gamma$ and $\alpha$.
\end{lemma}
\begin{proof}
    If $d(x_-,x_+)\ge \bdelta$ then the concatenation of the three geodesic segments $[\gamma_-,x_-][x_-,x_+][x_+,\gamma_+]$ is $\bdelta$-close to the geodesic $[\gamma_-,\gamma_+]$.
    Let $y\in [\gamma_-,\gamma_+]$ be $\bdelta$ close to $x_-$. Then,
    $d(\gamma,y)\ge d(\gamma,x_-)-\bdelta\ge d(\gamma,\alpha)-\bdelta$. 
    By exponential divergence we have $\length(\gamma)\ge \exp(\tfrac 1\bdelta d(\gamma,y))\ge \exp(\tfrac 1\bdelta d(\gamma,\alpha))$. 
    
    To get the result of the lemma decompose $\gamma = \gamma_1\dots\gamma_n$ into segments whose projections to $\alpha$ are at distance $\approx \bdelta$.
\end{proof}

Before proving \Cref{cor commensurated hyperbolic subgroup} we prove the following special case:

\begin{corollary}\label{commensurators of rigids}
Let $G$ be a residually-finite geometrically rigid hyperbolic group. If $G$ is a subgroup of the hyperbolic group $G_1$ then $[\Comm_{G_1}(G):G]<\infty.$
\end{corollary}

\begin{proof}[Proof of \Cref{commensurators of rigids}]
Let $G$ be a  residually-finite geometrically rigid hyperbolic group, and let $G_1$ be a hyperbolic group. Assume that $G\le G_1$ and $[\Comm_{G_1}G:G]=\infty$.
Let $S\subseteq S_1$ be finite generating sets for $G$ and $G_1$ respectively, and let $X,X_1$ be the corresponding Cayley graphs. 

Since $[\Comm_{G_1}(G):G]=\infty$, for every $n$ there exists $h_n\in \Comm_{G_1}(G)$ such that $d(X,h_nX)\ge n$.
Denote by $\phi_n$ the abstract commensurator $\phi_n:G\cap h_nGh_n\ii \to h_n\ii G h_n \cap G$ given by $\phi_n(g)=h_n\ii gh_n$.

By \cref{cor: AL}, $\AL(\phi_n,g)\le C\cdot \ell(g)$ for all $n$. Thus, to reach a contradiction, it remains to prove the following claim.

\begin{claim}
    For all $g\in G$ of infinite order, we have $\AL(\phi_n,g)\to \infty $ as $n\to \infty$.
\end{claim}

Let $g$ be some infinite order element in $G$. 
Let $o\in X$ be some vertex, 
The orbit $\gen{g}o\subset X$ is at bounded distance $\rho=\rho(g,o,\alpha)$ from a bi-infinite geodesic $\alpha$ in $X_1$.
Then, \begin{equation}\label{eq: dist between coset and geodesic}
    d(h_nX,\alpha)\ge d(h_nX,X)-\rho\ge n-\rho
\end{equation}
Let $k\in \bbN$ be such that $g^k\in G \cap h_nGh_n\ii$.
Then $g^k h_no\in h_nX$. 
Let $x_-,x_+$ be nearest point projections of $h_no,g^kh_no$ on $\alpha$. 
Denote $D = d(o,h_no)=d(g^ko,g^k h_no)$.
Then $d(h_no,\alpha)\le d(h_no,o)+d(o,\alpha )\le \rho+D$.
Therefore, $d(o,x_-o)\le d(o,h_no)+d(h_no,x_-)\le 2D+\rho$.
Similarly, $d(g^ko,x_+)\le \rho+2D$.
Therefore, \begin{equation}\label{eq: dist between projections}
    d(x_-,x_+)\ge d(o,g^ko)-2(\rho+2D).
\end{equation}

For large enough $k$, $d(x_-,x_+)\ge \bdelta$ and so we can apply \Cref{exp divergence from geodesic} to get
\begin{align*}
d_X(o,\phi(g^k)o)&=d_X(o,h_n\ii g^k h_no)\\
&=d_{h_nX}(h_no,g^kh_no)
\\
&\ge \exp(\tfrac 1\bdelta d(h_nX,\alpha)) d(x_-,x_+) &\text{by \cref{exp divergence from geodesic}}\\
&\ge \exp(\tfrac 1 \bdelta (n-\rho))\big(d(o,g^ko)-2(\rho+2D)\big) &\text{by \eqref{eq: dist between coset and geodesic} and \eqref{eq: dist between projections}}
\end{align*}
Diving by $k$ and taking a limit as $k\to \infty$ we get

$$
    \ell_G(\phi(g))\ge \exp(\tfrac 1 \bdelta (n-\rho)) \ell_{G_1}(g)
$$

The same inequality is true if we replace $g$ by the conjugate $tgt\ii$ for $t\in G$ (by replacing $g,o,\alpha$ by $tgt\ii,to, t\alpha$, and noting that $\rho(g,o,\alpha)=\rho(tgt\ii,to,t\alpha)$).
Therefore 
$$\AL(\phi,g) \ge \exp(\tfrac 1 \bdelta (n-\rho)) \ell_{G_1}(g)
$$
Indeed, $\AL(\phi_n,g)\to \infty$ as $n\to \infty$.
\end{proof}

With \Cref{commensurators of rigids} in place,
the rest of the proof of \Cref{cor commensurated hyperbolic subgroup}   follows that of \cite{lazarovich2023commensurated}. We sketch it here for completeness:

\begin{proof}[Proof of \Cref{cor commensurated hyperbolic subgroup}]
Assume that $G\le G_1$ are hyperbolic groups, $G$ is residually finite and $[\Comm_{G_1}(G):G]=\infty$.
Up to passing to finite index we may assume that $G$ is torsion-free.

\textbf{Special case: $G$ is 1-ended} If $G$ is 1-ended, then by \Cref{commensurators of rigids}, $G$ is not geometrically rigid. 
By \cite{bowditchjsj} (see also \cite{selajsj}), $G$ is either a surface group or $G$ has a non-trivial JSJ decomposition.
Assume for contradiction that $G$ has a non-trivial JSJ decomposition:  $G=\pi_1(\calG)$ where $\calG$ is a graph of groups with cyclic edge stabilizers. 
It follows from Bowditch's description of the JSJ decomposition, that for each $g_1\in \Comm_{G_1}(G)$ and each edge group $Z$ of $\calG$, there exists an edge group $Z'$ of $\calG$ such that $g_1Zg_1\ii,Z'$ are commensurable. 
Since there are finitely many $G$-conjugacy classes of edge groups, it follows that for each edge group $Z$ we have $[\Comm_{G_1}(Z):Z]=\infty$. This contradicts the well-known fact that in a hyperbolic group, every infinite cyclic subgroup has finite index in its relative commensurator.
Therefore, $G$ is a surface group.

\textbf{General case:}  
Let $G = A_1*\dots *A_n * F_r$ be its Grushko decomposition into freely indecomposable subgroups and a free part. 
Since $G$ is torsion-free, each $A_i$ is 1-ended.
It suffices to show that each $A_i$ is a surface group.

For each $g_1\in \Comm_{G_1}(G)$ the conjugate $g_1A_ig_1\ii$ is commensurable with some conjugate $gA_jg\ii$ for some $g\in G$ and $j\in\{1,\dots,n\}$.
It follows that $[\Comm_{G_1}(A_i):A_i]=\infty$.
The group $A_i\le G\le G_1$ is a 1-ended residually finite hyperbolic group of a hyperbolic group. So by the special case, $A_i$ is a surface group.
\end{proof}

\bibliographystyle{alpha}
\bibliography{biblio}
\end{document}

%% file: preamble.tex
\usepackage[leqno]{amsmath}
\usepackage{amssymb,amsfonts,xfrac,MnSymbol}
\usepackage{amsthm}
\usepackage{cite}
\usepackage[colorlinks=true, pdfstartview=FitV, linkcolor=blue, citecolor=blue, urlcolor=blue]{hyperref}
\usepackage{todonotes}
\usepackage{enumitem}
\usepackage{graphicx}
\usepackage{tikz-cd, tikz}
\usepackage{color}
\usepackage{import}
\usepackage[nameinlink,capitalize,noabbrev]{cleveref}
\usepackage{comment}
\usepackage{thm-restate}

\numberwithin{equation}{section}

\newtheorem{theorem}{Theorem}[section]
\newtheorem{claim}[theorem]{Claim}
\newtheorem{proposition}[theorem]{Proposition}
\newtheorem{lemma}[theorem]{Lemma}
\newtheorem{corollary}[theorem]{Corollary}
\newtheorem{conjecture}[theorem]{Conjecture}

\newtheorem*{theorem*}{Theorem}
\newtheorem*{claim*}{Claim}
\newtheorem*{proposition*}{Proposition}
\newtheorem*{lemma*}{Lemma}
\newtheorem*{corollary*}{Corollary}

\theoremstyle{definition}
\newtheorem{definition}[theorem]{Definition}

\newtheorem{remark}[theorem]{Remark}

\newtheorem*{definition*}{Definition}
\newtheorem*{observation*}{Observation}
\newtheorem*{remark*}{Remark}
\newtheorem*{example*}{Example}
\newtheorem*{question*}{Question}
\newtheorem*{exercise*}{Exercise}
\newtheorem*{fact*}{Fact}
\newtheorem*{notation*}{Notation}

\newcommand{\bbH}{\mathbb{H}}

\newcommand{\bbN}{\mathbb{N}}

\newcommand{\bbR}{\mathbb{R}}
\newcommand{\bbS}{\mathbb{S}}

\newcommand{\bbZ}{\mathbb{Z}}

\newcommand{\bfH}{\mathbf{H}}

\newcommand{\bfc}{\mathbf{c}}
\newcommand{\bfd}{\mathbf{d}}

\newcommand{\calB}{\mathcal{B}}

\newcommand{\calF}{\mathcal{F}}
\newcommand{\calG}{\mathcal{G}}

\newcommand{\calK}{\mathcal{K}}
\newcommand{\calL}{\mathcal{L}}

\newcommand{\calN}{\mathcal{N}}

\newcommand{\calS}{\mathcal{S}}

\newcommand{\actson}{\curvearrowright}

\newcommand{\inj}{\hookrightarrow}

\newcommand{\homeo}{\cong}

\newcommand{\ii}{^{-1}}

\newcommand{\gen}[1]{\left< #1 \right>}

\newcommand{\tild}[1]{\widetilde{#1}}
\newcommand{\bdelta}{{\boldsymbol\delta}}
\newcommand{\ssm}{\smallsetminus}

\newcommand{\nir}[1]{\todo[inline, color=magenta]{#1}}

\DeclareMathOperator{\SL}{SL}

\DeclareMathOperator{\SO}{SO}

\DeclareMathOperator{\id}{id}
\DeclareMathOperator{\Aut}{Aut}

\DeclareMathOperator{\Out}{Out}

\DeclareMathOperator{\diam}{diam}

\DeclareMathOperator{\Comm}{Comm}

\DeclareMathOperator{\conv}{Conv}
\DeclareMathOperator{\supp}{supp}

\DeclareMathOperator{\length}{length}

\DeclareMathOperator{\AD}{AvgDist}
\DeclareMathOperator{\TD}{TotalDist}
\DeclareMathOperator{\AL}{\ell_{Avg}}

\DeclareMathOperator{\vol}{vol}
\DeclareMathOperator{\Vol}{Vol}
\DeclareMathOperator{\Rips}{Rips}
\DeclareMathOperator{\MP}{MinPrim}
\DeclareMathOperator{\injrad}{InjRad}
\DeclareMathOperator{\MI}{MI}
\DeclareMathOperator{\MIG}{MIG}
